\documentclass[12pt,amstex]{amsart}
\usepackage[T1]{fontenc}
\usepackage[latin9]{inputenc}
\usepackage{multirow}
\usepackage{array}
\usepackage{amsmath, amssymb}
\usepackage{amsthm} 
 \usepackage{bm}

\usepackage{a4wide}
\usepackage[ngerman, english]{babel}
\usepackage{newtxmath,newtxtext}
\usepackage{setspace}
\usepackage{amsmath,mathtools}
\usepackage{amsfonts}

\theoremstyle{plain} 
\newtheorem{theorem}{Theorem}[section]
\newtheorem{lemma}[theorem]{Lemma}
\newtheorem{corollary}[theorem]{Corollary}
\newtheorem{proposition}[theorem]{Proposition}

\newtheorem{theor}{Theorem}

\theoremstyle{definition}

\newtheorem{remark}[theorem]{Remark}
\newtheorem{example}[theorem]{Example}

 \usepackage{caption}

\usepackage{color}
\usepackage{graphicx,color}
\usepackage{amsmath, amssymb, graphics}

\graphicspath{{./Figures/}}

\DeclareMathOperator{\tr}{tr}
\DeclareMathOperator{\arcsinh}{arcsinh}
\DeclareMathOperator{\arccosh}{arccosh}

\begin{document}
\title[Wick rotations of solutions to three equations]{Wick rotations of solutions to the minimal surface equation, the zero mean curvature equation and the Born-Infeld equation}
  \author[S. Akamine R. Singh]{Shintaro Akamine and Rahul Kumar Singh}
   \address{Shintaro Akamine: Graduate School of Mathematics Kyushu University\\
744 Motooka, Nishi-ku,\
Fukuoka 819-0395, Japan}
 \email{s-akamine@math.kyushu-u.ac.jp}
 
  \address{Rahul Kumar Singh: School of Mathematical Sciences, National Institute of Science Education and Research, HBNI, Bhubaneshwar, Khurda, Odisha--752050, India}
 \email{rahul412@niser.ac.in}



\begin{abstract}In this paper we investigate relations between solutions to the minimal surface equation in Euclidean $3$-space $\mathbb{E}^3$, the zero mean curvature equation in Lorentz-Minkowski $3$-space $\mathbb{L}^3$ and the Born-Infeld equation under Wick rotations. We prove that the existence conditions of real solutions and imaginary solutions after Wick rotations are written by symmetries of solutions, and reveal how real and imaginary solutions are transformed under Wick rotations. We also give a transformation theory for zero mean curvature surfaces containing lightlike lines with some symmetries.  As an application, we give new correspondences among some solutions to the above equations by using the non-commutativity between Wick rotations and isometries in the ambient space.\end{abstract}

\maketitle

\section{Introduction}\label{Sec.1}
In this paper we study geometric relations of real analytic solutions to the following three equations
\begin{align}
(1+f_y^2)f_{xx}-2f_xf_yf_{xy}+(1+f_x^2)f_{yy}=0,\label{eq:M} \\
(1-g_y^2)g_{xx}+2g_xg_yg_{xy}+(1-g_x^2)g_{yy}=0,\label{eq:ZMC}\\
(1-h_t^2)h_{xx}+2h_th_xh_{tx}-(1+h_x^2)h_{tt}=0.\label{eq:BI}
\end{align}

The equation (\ref{eq:M}) is called the {\it minimal surface equation} which is the equation of minimal graphs over a domain of the $xy$-plane in Euclidean 3-space $\mathbb{E}^3=\mathbb{E}^3(x,y,z)$. The second equation (\ref{eq:ZMC}) is called the {\it zero mean curvature equation}, which is the equation of graphs with zero mean curvature over a domain of the spacelike $xy$-plane in Lorentz-Minkowski 3-space $\mathbb{L}^3=\mathbb{L}^3(t,x,y)$. The graph of $g$ is spacelike if $1-g_x^2-g_y^2>0$, timelike if $1-g_x^2-g_y^2<0$ and lightlike if $1-g_x^2-g_y^2=0$. The third equation (\ref{eq:BI}) is called the {\it Born-Infeld equation}, which is the equation of graphs with zero mean curvature over a domain of the timelike $tx$-plane in $\mathbb{L}^3$. The equation (\ref{eq:BI}) also appears in a geometric nonlinear theory of electromagnetism, which is known as Born-Infeld model introduced by Born and Infeld \cite{BI}. A surface in $\mathbb{L}^3$ whose mean curvature vanishes identically is called a {\it zero mean curvature surface}, and such surface can be written as the graph of a solution to (\ref{eq:ZMC}) or (\ref{eq:BI}) after a rigid motion in $\mathbb{L}^3$. A spacelike (resp.\ timelike) zero mean curvature surface is called a {\it maximal surface} (resp.\ {\it timelike minimal surface}).

It is known that there is a duality between solutions to (\ref{eq:M}) and spacelike solutions to (\ref{eq:ZMC}) called the {\it Calabi's correspondence} \cite{Calabi} as follows. 
Let $z=f(x,y)$ be a minimal graph over a simply-connected domain. Since the equation (\ref{eq:M}) for $f$ is equivalent to
\[
\frac{\partial}{\partial x}\left(\frac{f_x}{\sqrt{1+f_x^2+f_y^2}}\right)+\frac{\partial}{\partial y}\left(\frac{f_y}{\sqrt{1+f_x^2+f_y^2}}\right)=0,
\]
we can find a spacelike solution to (\ref{eq:ZMC}) on the same domain such that 
\[
(g_x,g_y)=\frac{(-f_y,f_x)}{\sqrt{1+f_x^2+f_y^2}},\quad 1-g_x^2-g_y^2>0.
\]
We can recover a solution to (\ref{eq:M}) from a spacelike solution to (\ref{eq:ZMC}) by the same procedure. On the other hand, there are some solutions to (\ref{eq:M}) and (\ref{eq:ZMC}) which do not correspond by the Calabi's correspondence but resemble each other as pointed out in \cite{FujimoriETAL1,Kobayashi}. For example, the solution to (\ref{eq:M}) 
\begin{equation}\label{eq:Scherk}
z=f(x,y)=\log{\left(\frac{\cos{x}}{\cos{y}}\right)}
\end{equation}
is the graph of the classical Scherk minimal surface in $\mathbb{E}^3$. On the other hand, 
\begin{equation}\label{eq:ZMC_Scherk}
t=g(x,y)=\log{\left(\frac{\cosh{x}}{\cosh{y}}\right)}
\end{equation}
is a solution to (\ref{eq:ZMC}) in $\mathbb{L}^3(t,x,y)$ which is an entire graph found by Kobayashi \cite{Kobayashi} having all causal characters (Figure 1, center). This surface can be obtained by replacing $x$ and $y$ by $ix$ and $iy$ in (\ref{eq:Scherk}), and by identifying $z$ with $t$. Moreover, if we replace only $y$ by $iy$ in (\ref{eq:Scherk}), and by identifying $(x,y,z)$ with $(\tilde{x},\tilde{t},\tilde{y})$, we have the following solution to (\ref{eq:BI}) in $\mathbb{L}^3(\tilde{t},\tilde{x},\tilde{y})$
\[
\tilde{y}=h(\tilde{t},\tilde{x})=\log{\left(\frac{\cos{\tilde{x}}}{\cosh{\tilde{t}}}\right)}.
\]

\begin{figure}[!h]
\begin{center}
\begin{tabular}{c}
\hspace{+0.1cm}
\begin{minipage}{0.3\hsize}
\begin{center}
\vspace{-0.3cm}
\includegraphics[clip,scale=0.23,bb=0 0 350 380]{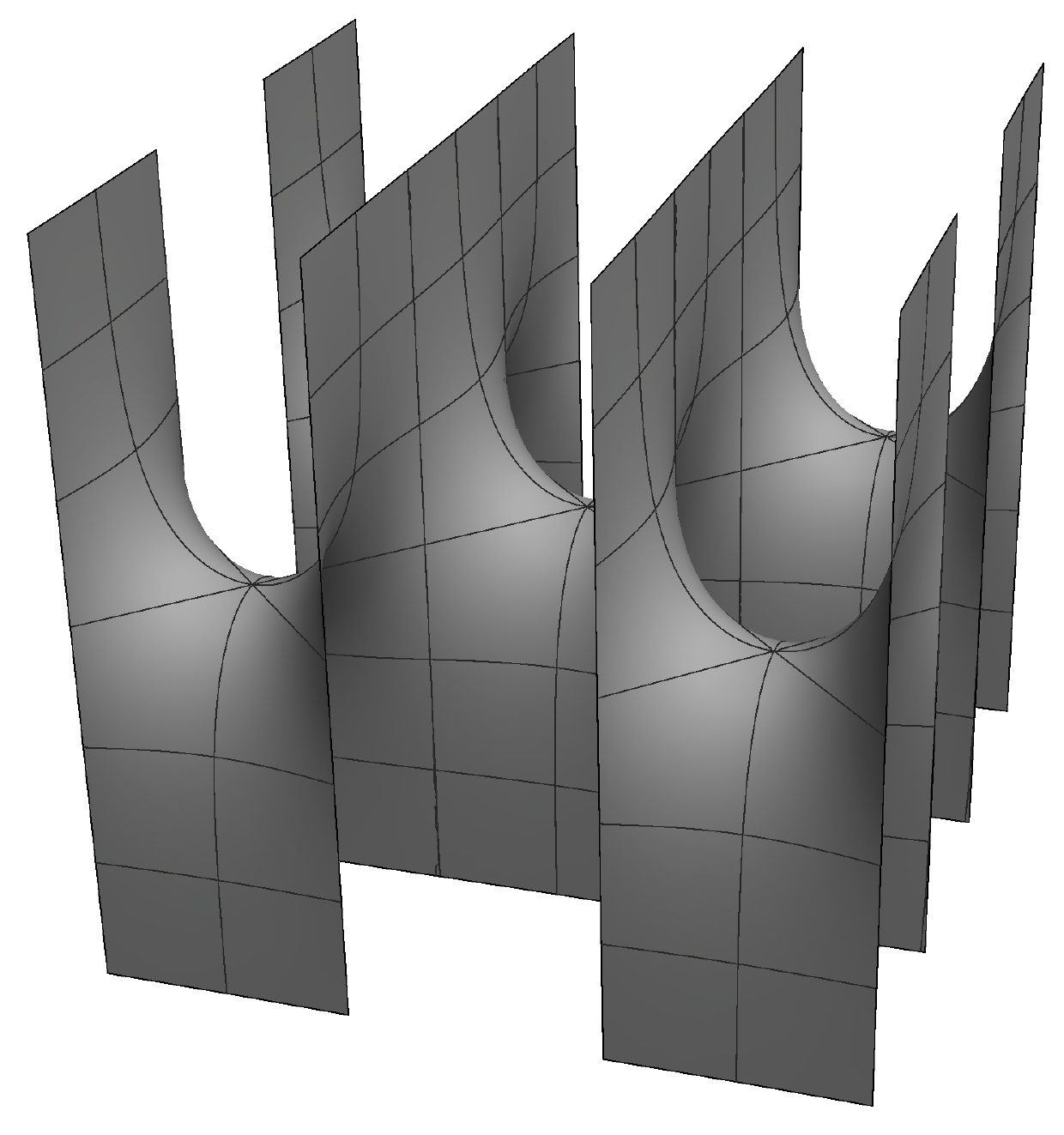}
\vspace{0.5cm}
\end{center}
\end{minipage}
\hspace{0.0cm}
\begin{minipage}{0.3\hsize}
\begin{center}
\vspace{-1.3cm}
\includegraphics[clip,scale=0.28,bb=0 0 350 380]{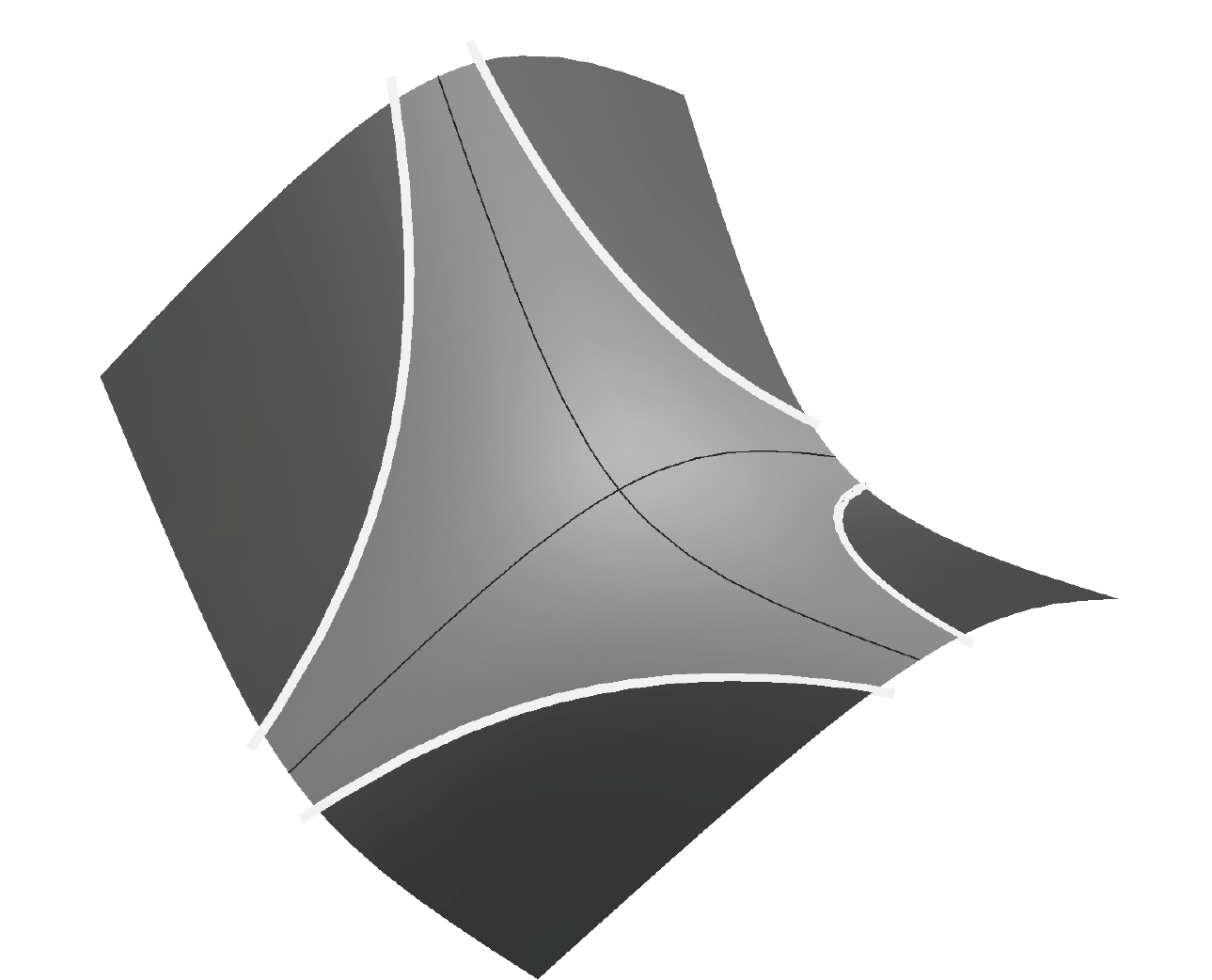}
\vspace{0.3cm}
\end{center}
\end{minipage}

\hspace{-0.1cm}
\begin{minipage}{0.3\hsize}
\begin{center}
\vspace{-0.3cm}
\includegraphics[clip,scale=0.23,bb=0 0 350 380]{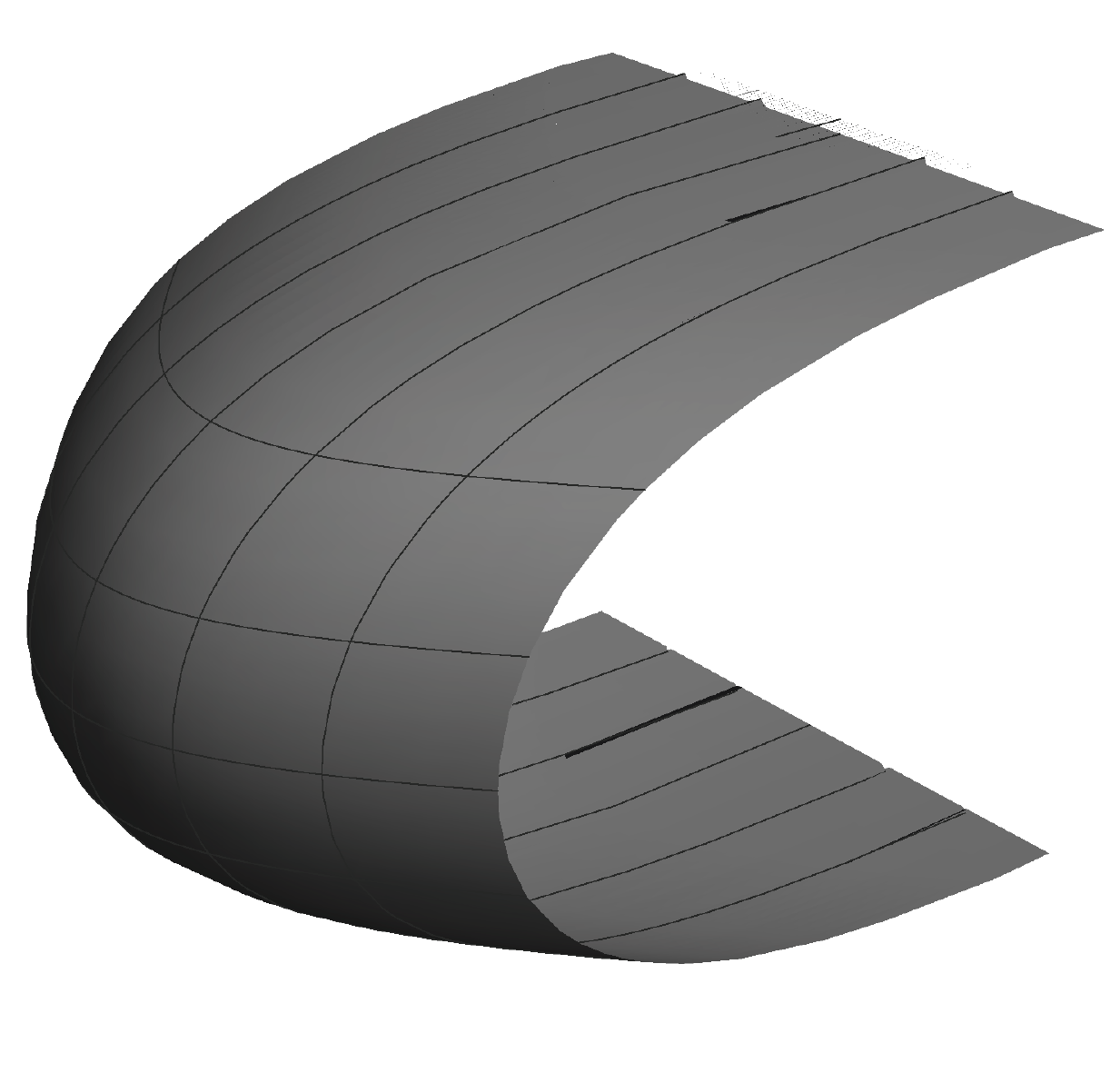}
\vspace{0.5cm}
\end{center}
\end{minipage}

\end{tabular}
\end{center}
\vspace{-0.3cm}
\caption{From left to right: The doubly periodic Scherk minimal surface, Scherk type zero mean curvature surface in \cite{Kobayashi} and the corresponding solution to (\ref{eq:BI}).}\label{Fig1}

\end{figure}

 In general, solutions to the equations (\ref{eq:M}), (\ref{eq:ZMC}) and (\ref{eq:BI}) are related  by changing parameters called {\it Wick rotations}. In $ 1954 $, the physicist Wick \cite{GCW} argued that one is allowed to consider the wave function for imaginary values of $ t $, i.e.,  replacing the real time variable $ t $  by imaginary time variable $ it $. This method of changing a real parameter to an imaginary parameter  is what is known as {\it Wick rotation}. It also motivates the observation that the Minkowski metric $ -dt^2+dx^2+dy^2 $ and the Euclidean metric $ dt^2+dx^2+dy^2 $ are equivalent if the time component of either are allowed to have imaginary values. In general, one can use the concept of Wick rotation as a method of finding  solutions to a problem in Minkowski space from solutions to a related problem in Euclidean space.  In past several authors have used this technique of Wick rotation in many different contexts \cite{DS,GI,Kamien,KKSY,MVV}. In our setting, as pointed out in \cite[Section $5$]{MVV}, the problem discussed so far to generate new solutions by using Wick rotations is that solutions may be complex valued, in general. In this paper, we give criteria for existence of real and imaginary solutions after Wick rotations, and study geometric properties of these correspondences as follows (Theorem \ref{thm:BI and M}).
\begin{theor}\label{maintheorem1}
Let $z=f(x,y)$ be a solution to (\ref{eq:M}) without umbilic points on $y=0$. The following statements hold.
\begin{itemize}
\item[(i)] If $f$ is even with respect to the $y$-axis, then the graph $\tilde{y}=f(\tilde{x},i\tilde{t})$ in $\mathbb{L}^3(\tilde{t},\tilde{x},\tilde{y})$ is a timelike solution to (\ref{eq:BI}) with negative Gaussian curvature near $\tilde{t}=y=0$, where $\tilde{t}=y$, $\tilde{x}=x$ and $\tilde{y}=z$.
\item[(ii)]If $f$ is odd with respect to the $y$-axis, then the graph $\tilde{y}=-if(\tilde{t},i\tilde{x})$ in $\mathbb{L}^3(\tilde{t},\tilde{x},\tilde{y})$ is a timelike solution to (\ref{eq:BI}) with positive Gaussian curvature near $\tilde{x}=y=0$, where $\tilde{t}=x$, $\tilde{x}=y$ and $\tilde{y}=z$.
\end{itemize}
\end{theor}
We also prove similar correspondences for solutions to (\ref{eq:ZMC}) and (\ref{eq:BI}) in Theorem \ref{thm:BI and ZMC}, and for solutions to (\ref{eq:M}) and (\ref{eq:ZMC}) in Theorem \ref{thm:M and ZMC}.

Moreover, we study in Section \ref{Sec.4} Wick rotations starting from {\it lightlike points}, that is, points on which the metric of a surface is degenerate. In general, maximal surfaces and timelike minimal surfaces in $\mathbb{L}^3$ have {\it singular points}, on which surfaces are not immersed. Such singular points appear as lightlike points and most of singular points have been studied by Weierstrass type representation formulas (see, for example, \cite{ER,FLS,FujimoriETAL3,Kobayashi,UY06}) and the singular Bj\"orling formula (cf.\ \cite{KKSY,KY}) on isothermal coordinates. However it is known that there is a case that lightlike points consist of a lightlike line, on which the isothermal coordinates break down (see \cite[Section 1]{FujimoriETAL1} and also \cite[Lemma 3.2]{KKSY}, \cite[Corollary 3.3]{KY}). Recently such surfaces have been studied intensively. In \cite{FujimoriETAL1}, zero mean curvature surfaces with lightlike lines were categorized into the following six classes (the definition of each class is given in Section \ref{Sec.4})
\[
\alpha^+,\quad \alpha^0_I,\quad \alpha^0_{II},\quad \alpha^-_{I},\quad \alpha^-_{II},\quad \alpha^-_{III},
\]
and many examples were given in \cite{A1,FujimoriETAL1,FujimoriETAL2,FujimoriETAL3,UY}. In particular, for each class as above, the existence of a zero mean curvature surface which can have any possible causal character along a lightlike line was proved in \cite{UY}. However, since one cannot take isothermal coordinates near lightlike lines, there is no known explcit representation formula for such surfaces. In this paper, we give a transformation theory for zero mean curvature surfaces with lightlike lines via Wick rotations. More precisely, we prove the following transformation method via Wick rotations (Theorem \ref{thm:transformations of L-line}).

\begin{theor}\label{maintheorem2}
Let $t=g(x,y)$ be a real analytic solution to (\ref{eq:ZMC}) with a lightlike line segment $L$ which contains the origin in $\mathbb{L}^3$.
\begin{itemize}
\item[(i)] If $g$ is even with respect to the $x$-axis, then the solution $\tilde{y}=h(\tilde{t},\tilde{x})=g(i\tilde{x},\tilde{t})$ in $\mathbb{L}^3(\tilde{t},\tilde{x},\tilde{y})$ to (\ref{eq:BI}) also has a lightlike line segment $\tilde{L}$. Moreover these two solutions belong to the same class as above along $L$ and $\tilde{L}$, where $\tilde{t}=y$, $\tilde{x}=x$ and $\tilde{y}=t$.
\item[(ii)] If $g$ is odd with respect to the $x$-axis, then the solution $\tilde{y}=h(\tilde{t},\tilde{x})=-ig(i\tilde{t},\tilde{x})$ in $\mathbb{L}^3(\tilde{t},\tilde{x},\tilde{y})$ to (\ref{eq:BI}) also has a lightlike line segment $\tilde{L}$, where $\tilde{t}=x$, $\tilde{x}=y$ and $\tilde{y}=t$. Moreover each of solutions belongs to the class $\alpha^+$, $\alpha^0_{I}$ or $\alpha^-_{I}$. A solution of $\alpha^+$ (resp.\ $\alpha^-_{I}$) type is transformed to a solution of $\alpha^-_{I}$ (resp.\ $\alpha^+$) types, and a graph of type $\alpha^0_{I}$ is transformed to a graph of type $\alpha^0_{I}$.
\end{itemize}
\end{theor}
In Section \ref{Sec.5}, by using Theorem \ref{maintheorem2} and the non-commutativity of Wick rotations and isometries in the ambient space, we give new relations among many examples, most of them were constructed in \cite{FujimoriETAL1}.

\section{Preliminaries\label{Sec.2}}

In this paper, we deal only with real analytic immersions and real analytic graphs. We denote by $\mathbb{L}^3=\mathbb{L}^3(t,x,y)$ the Lorentz-Minkowski 3-space with the metric $\langle , \rangle =-(dt)^2+(dx)^2+(dy)^2$, where $(t, x, y)$ are the canonical coordinates. An immersion $ X\colon \Omega\rightarrow \mathbb{L}^3$ of a domain $\Omega \subset \mathbb{R}^2$ into $\mathbb{L}^3$ is called {\it spacelike} (resp.\ {\it timelike}, {\it lightlike}) at a point $p\in \Omega$ if its {\it first fundamental form} $\mathrm{I}=X^*{\langle , \rangle}$ is Riemannian (resp.\ Lorentzian, degenerate) at $p$. For a spacelike (or timelike) immersion, we can take a timelike (or spacelike) unit normal vector field $\nu$. Let $ \nabla $ denote the Levi-Civita connection on $ \mathbb{L}^3 $, and suppose $\xi$ and $\eta$ are  smooth vector fields on $\Omega$. Then the {\it shape operator} (or the {\it Weingarten map}) $S$ of $X$ and the {\it second fundamental form} $\mathrm{II}$ are defined by
 $$ dX(S(\xi))=-\nabla_{\xi}\nu,\quad {\rm II}(\xi,\eta)=\langle \nabla_{dX(\xi)}dX(\eta), \nu \rangle.$$
The {\it mean curvature} $H$ and the {\it Gaussian curvature} $K$ of the surface $X$ are defined by
 $$ H=\frac{\varepsilon}{2}\tr{S},\quad K=\varepsilon\det{S},$$
 where 
\begin{equation*}
\varepsilon = \langle \nu,\nu \rangle=
\begin{cases}
1  & \text{if $X$ is timelike,}\\
-1 & \text{if $X$ is spacelike.}
\end{cases}
\end{equation*} 

An eigenvalue of the shape operator $S$ is called a {\it principal curvature} of $X$. When an immersion $X$ is spacelike, the shape operator $S$ is symmetric. Hence we can always take real principal curvatures at any point, and a point on which two principal curvatures are equal is called an {\it umbilic point}. On the other hand, the shape operator $S$ for a timelike immersion is not always diagonalizable even over the complex number field $\mathbb{C}$. In this case, a point on which $S$ has same real principal curvatures is called an {\it umbilic point}, and a point on which $S$ is non-diagonalizable over $\mathbb{C}$ is called a {\it quasi-umbilic point}, see \cite{Clelland} or \cite{A2} for details.

After a rigid motion in $\mathbb{L}^3$, any spacelike or timelike surface can be written as a graph over a domain in the spacelike $xy$-plane, which has the form 
      $$ X(x,y)=(g(x,y),x,y),$$ 
  or a graph over a domain in the timelike $tx$-plane, which has the form
      $$ Y(x,t)=(t,x,h(x,t)). $$ 
The first fundamental form for $X$ or $Y$, we denote them by $\mathrm{I}_g$ and $\mathrm{I}_h$, are
 \[
\mathrm{I_g}=(1-g_x^2)dx^2-2g_xg_ydxdy+(1-g_y^2)dy^2\quad \text{and}\quad \mathrm{I_h}=(-1+h_t^2)dt^2+2h_th_xdtdx+(1+h_x^2)dx^2.
\]

 The unit normal vector fields for each of the case are given by           
  \[ \nu_g(x,y)=\frac{(1,g_x,g_y)}{\sqrt{|1-g_x^2-g_y^2|}}\quad \text{and}\quad \nu_h(t,x)=\frac{(h_t,-h_x,1)}{\sqrt{|-1-h_x^2+h_t^2|}}.\]
Put $\varepsilon_g:=\langle \nu_g,\nu_g \rangle$ and $\varepsilon_h:=\langle \nu_h,\nu_h \rangle$. The second fundamental forms of $X$ and $Y$, we denote them by $\mathrm{II}_g$ and $\mathrm{II}_h$, are
\[
\mathrm{II}_g=-\left(\frac{g_{xx}}{W_g}dx^2+2\frac{g_{xy}}{W_g}dxdy+\frac{g_{yy}}{W_g}dy^2\right),\quad W_g:=\sqrt{-\varepsilon_g(1-g_x^2-g_y^2)},
\]

\[
\mathrm{II}_h=\left(\frac{h_{tt}}{W_h}dt^2+2\frac{h_{tx}}{W_h}dtdx+\frac{h_{xx}}{W_h}dx^2\right),\quad W_h:=\sqrt{-\varepsilon_h(1-h_t^2+h_x^2)}.
\]
The mean curvature $H_g$ and the Gaussian curvature $K_g$ of $X$ is written as
\begin{equation*}\label{eq:H_gK_g}
H_g=\frac{(1-g_y^2)g_{xx}+2g_xg_yg_{xy}+(1-g_x^2)g_{yy}}{2W_g^3}\quad \text{and}\quad K_g=-\frac{g_{xx}g_{yy}-g_{xy}^2}{W_g^4}.
\end{equation*}
While, the mean curvature $H_h$ and the Gaussian curvature $K_h$ of $Y$ is written as
\begin{equation*}\label{eq:H_hK_h}
H_h=\frac{(1-h_t^2)h_{xx}+2h_th_xh_{tx}-(1+h_x^2)h_{tt}}{2W_h^3}\quad \text{and}\quad K_h=-\frac{h_{tt}h_{xx}-h_{tx}^2}{W_h^4}.
\end{equation*}
Then the condition that the mean curvature vanishes identically leads to the equations (\ref{eq:ZMC}) and (\ref{eq:BI}), respectively. A surface in $\mathbb{L}^3$ whose mean curvature vanishes identically is called a {\it zero mean curvature surface}, and a spacelike (resp.\ timelike) zero mean curvature surface is called a {\it maximal surface} (resp.\ {\it timelike minimal surface}).

On the other hand, if one considers a graph $ Z(x,y)=(x,y,f(x,y)) $ in Euclidean 3-space $ \mathbb{E}^3=\mathbb{E}^3(x,y,z) $ over a domain in the $ xy $-plane, then the  mean curvature is written as
\[
H_f=\frac{(1+f_y^2)f_{xx}-2f_xf_yf_{xy}+(1+f_x^2)f_{yy}}{2W_f^3},\quad W_f:=\sqrt{1+f_x^2+f_y^2}
\]
Then $H_f$ vanishes identically, that is, $Z$ is a {\it minimal surface} if and only if $f$ satisfies \eqref{eq:M}.  The Gaussian curvature of $Z$ is written as
\begin{equation}\label{eq:K_f}
K_f=\frac{f_{xx}f_{yy}-f_{xy}^2}{W_f^4}.
\end{equation}
From Section \ref{Sec.3}, we discuss correspondences among minimal surfaces in $\mathbb{E}^3$ and zero mean curvature surfaces in $\mathbb{L}^3$ via Wick rotations as explained in Introduction.

\section{Real and imaginary solutions}\label{Sec.3}
In this section we give geometric relationships among solutions to the three equations (\ref{eq:M}), (\ref{eq:ZMC}) and (\ref{eq:BI}) under Wick rotations mentioned in Introduction. First we prove a necessary and sufficient condition for getting the real or imaginary solution after the Wick rotation of a solution.
\begin{lemma}\label{lemma:symmetry}
Let $\varphi=\varphi(x_1,x_2)$ be a real analytic function, and consider its Wick rotation $\psi(x_1,x_2)=\varphi(ix_1,x_2)$, where $x_1$ and $x_2$ are two of the canonical coordinates in $\mathbb{E}^3$ or $\mathbb{L}^3$. Then $\psi$ is real (resp.\ imaginary) valued if and only if $\varphi$ is even (resp.\ odd) with respect to the $x_1$-axis. Moreover $\psi$ is also even (resp.\ odd) with respect to the $x_1$-axis.
\end{lemma}
\begin{proof}
Taking the following expansion near $x_1=0$
\begin{equation*}
\varphi(x_1,x_2)=\sum_{n=0}^{\infty}a_n(x_2)x_1^n,
\end{equation*}
$\psi$ can be written as
\begin{equation*}
\psi(x_1,x_2)=\sum_{n=0}^{\infty}a_{2n}(x_2)(-1)^nx_1^{2n}+i\sum_{m=0}^{\infty}a_{2m+1}(x_2)(-1)^mx_1^{2m+1}.
\end{equation*}
Therefore $\psi$ is real (resp.\ imaginary) valued if and only if the odd (resp.\ even) terms vanish, which proves the desired result.
\end{proof}

 \subsection{Geometric properties of transformations between solutions to (\ref{eq:M}) and (\ref{eq:BI})}
By the correspondence between solutions to (\ref{eq:M}) and (\ref{eq:BI}) under Wick rotations which was mentioned in \cite{D,Kamien} and Lemma \ref{lemma:symmetry}, we have the following.

\begin{proposition}\label{prop:BI and M_even}
For a solution $z=f(x,y)$ to (\ref{eq:M}) with even symmetry with respect to the $y$-axis, $\tilde{y}=f(\tilde{x},i\tilde{t})$ is a real solution to (\ref{eq:BI}) in $\mathbb{L}^3(\tilde{t},\tilde{x},\tilde{y})$. Conversely, for a solution $\tilde{y}=h(\tilde{t},\tilde{x})$ to (\ref{eq:BI}) in $\mathbb{L}^3(\tilde{t},\tilde{x},\tilde{y})$ with even symmetry with respect to the $\tilde{t}$-axis, $z=h(iy,x)$ is a real solution to (\ref{eq:M}) in $\mathbb{E}^3(x,y,z)$, where $\tilde{t}=y$, $\tilde{x}=x$ and $\tilde{y}=z$.
\end{proposition}

From imaginary solutions obtained by Wick rotations, we can also construct real solutions as follows:
 \begin{proposition}\label{prop:BI and M_odd}
For a solution $z=f(x,y)$ to (\ref{eq:M}) with odd symmetry to the $y$-axis, $\tilde{y}=-if(\tilde{t},i\tilde{x})$ is a real solution to (\ref{eq:BI}) in $\mathbb{L}^3(\tilde{t},\tilde{x},\tilde{y})$. Conversely, for a solution $\tilde{y}=h(\tilde{t},\tilde{x})$ to (\ref{eq:BI}) in $\mathbb{L}^3(\tilde{t},\tilde{x},\tilde{y})$ with odd symmetry to the $\tilde{x}$-axis, $z=-ih(x,iy)$ is a real solution to (\ref{eq:M}) in $\mathbb{E}^3(x,y,z)$, where $\tilde{t}=x$, $\tilde{x}=y$ and $\tilde{y}=z$.
\end{proposition}
\begin{proof}
We prove only the first part. If we set $h(\tilde{t},\tilde{x})=-if(\tilde{t},i\tilde{x})$, then $\tilde{y}=h(\tilde{t},\tilde{x})$ is real valued by Lemma \ref{lemma:symmetry}. Since
\[
h_{\tilde{t}}=-if_x,\quad h_{\tilde{x}}=f_y,\quad h_{\tilde{t}\tilde{t}}=-if_{xx},\quad h_{\tilde{t}\tilde{x}}=f_{xy},\quad h_{\tilde{x}\tilde{x}}=if_{yy},
\]
we obtain the relation
\begin{align*}
&\left[(1-h_{\tilde{t}}^2)h_{\tilde{x}\tilde{x}}+2h_{\tilde{t}}h_{\tilde{x}}h_{\tilde{t}\tilde{x}}-(1+h_{\tilde{x}}^2)h_{\tilde{t}\tilde{t}}\right](\tilde{t},\tilde{x})\\
=&i\left[(1+f_y^2)f_{xx}-2f_xf_yf_{xy}+(1+f_x^2)f_{yy}\right](x,iy),
\end{align*}
where $\tilde{t}=x$ and $\tilde{x}=y$.  By (\ref{eq:M}) and the identity theorem, the right hand side of the above equation is equal to zero for any $(x,y)$. Hence $h$ satisfies (\ref{eq:BI}).
\end{proof}

Solutions to (\ref{eq:BI}) obtained from solutions to (\ref{eq:M}) as Propositions \ref{prop:BI and M_even} and \ref{prop:BI and M_odd} are zero mean curvature graphs over a domain of the timelike $tx$-plane in $\mathbb{L}^3(t,x,y)$. From now on, we study geometric properties of these zero mean curvature surfaces.
\begin{theorem}\label{thm:BI and M}
Let $z=f(x,y)$ be a solution to (\ref{eq:M}) without umbilic points on $y=0$. The following statements hold.
\begin{itemize}
\item[(i)] If $f$ is even with respect to the $y$-axis, then the graph $\tilde{y}=f(\tilde{x},i\tilde{t})$ in $\mathbb{L}^3(\tilde{t},\tilde{x},\tilde{y})$ is a timelike minimal surface with negative Gaussian curvature near $\tilde{t}=y=0$, where $\tilde{t}=y$, $\tilde{x}=x$ and $\tilde{y}=z$.
\item[(ii)]If $f$ is odd with respect to the $y$-axis, then the graph $\tilde{y}=-if(\tilde{t},i\tilde{x})$ in $\mathbb{L}^3(\tilde{t},\tilde{x},\tilde{y})$ is a timelike minimal surface with positive Gaussian curvature near $\tilde{x}=y=0$, where $\tilde{t}=x$, $\tilde{x}=y$ and $\tilde{y}=z$.
\end{itemize}
\end{theorem}
\begin{proof}
First we prove (i). We set $h_1(\tilde{t},\tilde{x}):=f(\tilde{x},i\tilde{t})$. Since the first fundamental form of the graph $h_1$ at $(\tilde{t},\tilde{x})=(y,x)$ is
\[
-(1+f_y^2(\tilde{x},i\tilde{t}))d\tilde{t}^2+2if_x(\tilde{x},i\tilde{t})f_y(\tilde{x},i\tilde{t})d\tilde{t}d\tilde{x}+(1+f_x^2(\tilde{x},i\tilde{t}))d\tilde{x}^2,
\]
the graph is timelike near $\tilde{t}=y=0$. If we set $W_f=\sqrt{1+f_x^2+f_y^2}$, then the unit normal vector field is written as
\[
(if_y(\tilde{x},i\tilde{t}),-f_x(\tilde{x},i\tilde{t}),1)/{W_f(\tilde{x},i\tilde{t})}.
\]
By using it, the second fundamental is computed as
\[
-\frac{f_{yy}(\tilde{x},i\tilde{t})}{W_f(\tilde{x},i\tilde{t})}d\tilde{t}^2+2\frac{if_{xy}(\tilde{x},i\tilde{t}))}{W_f(\tilde{x},i\tilde{t})}d\tilde{t}d\tilde{x}+\frac{f_{xx}(\tilde{x},i\tilde{t})}{W_f(\tilde{x},i\tilde{t})}d\tilde{x}^2,
\]
and hence the Gaussian curvature of $h_1$, we denote it by $K_{h_1}$,  is
\begin{equation}\label{eq:Gauusian_BI_even_time}
K_{h_1}(\tilde{t},\tilde{x})=\left(\frac{f_{xx}f_{yy}-f_{xy}^2}{W_f^4}\right)(\tilde{x},i\tilde{t}).
\end{equation}

On the other hand, the Gaussian curvature of the graph $f$, we denoted it by $K_f$, is written as (\ref{eq:K_f}). Since $f$ is minimal and $(x,0)$ is not an umbilic point, $K_f<0$. By (\ref{eq:K_f}) and (\ref{eq:Gauusian_BI_even_time}), we obtain (i). Next, we set $h_2(\tilde{t},\tilde{x}):=-if(\tilde{t},i\tilde{x})$ and prove (ii). Since the first fundamental form of the graph $h_2$ at $(\tilde{t},\tilde{x})=(x,y)$ is
\[
-(1+f_x^2(\tilde{t},i\tilde{x}))d\tilde{t}^2-2if_x(\tilde{x},i\tilde{t})f_y(\tilde{t},i\tilde{x})d\tilde{t}d\tilde{x}+(1+f_y^2(\tilde{t},i\tilde{x}))d\tilde{x}^2,
\]
the graph is also timelike near $\tilde{x}=y=0$. The Gaussian curvature of $h_2$, we denote it by $K_{h_2}$,  is
\begin{equation}\label{eq:Gauusian_BI_odd_time}
K_{h_2}(\tilde{t},\tilde{x})=-\left(\frac{f_{xx}f_{yy}-f_{xy}^2}{W^4}\right)(\tilde{t},i\tilde{x}).
\end{equation}
Therefore, the graph has positive Gaussian curvature near $\tilde{x}=y=0$ by (\ref{eq:K_f}) and (\ref{eq:Gauusian_BI_odd_time}).
\end{proof}

\begin{remark}
An umbilic point $(\tilde{t},\tilde{x})=(y,x)=(0,x)$ of a minimal graph $f$ can be transformed into an umbilic point of the graph $h_1$ or $h_2$ defined as in the proof of Theorem \ref{thm:BI and M}. In fact, for the case (i) in Theorem \ref{thm:BI and M}, $f$ satisfies $f_{xy}(x,0)=0$ by the even symmetry of $f$ with respect to $y$. Hence, $K_{h_1}(0,\tilde{x})=K_{f}(x,0)=0$ is equivalent to the condition $f_{xx}(x,0)f_{yy}(x,0)=0$. If We assume that $f_{xx}(x,0)=0$, we have $f_{yy}(x,0)=0$ by (\ref{eq:M}), and the converse is also true. Therefore the second fundamental from of $h_1$ vanishes at $(\tilde{t},\tilde{x})=(0,x)$, which proves that $(\tilde{t},\tilde{x})=(0,x)$ is an umbilic point of $h_1$. Similarly, for the case (ii), we can prove that $(\tilde{t},\tilde{x})=(x,0)$ is also an umbilic point of $h_2$ in the proof of Theorem \ref{thm:BI and M}. Therefore quasi-umbilic points do not appear on the center of symmetries.
\end{remark}

In Theorem \ref{thm:BI and M}, we saw that minimal surfaces in $\mathbb{E}^3$ with even (resp.\ odd) symmetry with respect to an axis correspond to timelike minimal surfaces in $\mathbb{L}^3$ with negative (resp.\ positive) Gaussian curvature. As pointed out in \cite{A2}, the diagonalizability of the shape operator of a timelike minimal surface corresponds to the sign of the Gaussian curvature. As a corollary of Theorem \ref{thm:BI and M}, we have a result about relations between symmetries and diagonalizability of the shape operator of timelike minimal surfaces.
\begin{corollary}\label{cor:symmetry1}
Away from flat points, a timelike minimal graph $y=h(t,x)$ with even (resp.\ odd) symmetry with respect to the $t$-axis (resp.\ $x$-axis) has real (resp.\ complex) principal curvatures near the axis.
\end{corollary}
\begin{proof}
By Lemma \ref{lemma:symmetry} and Proposition \ref{prop:BI and M_even} (resp.\ Proposition \ref{prop:BI and M_odd}), the Wick rotated solution $\tilde{z}=f(\tilde{x},\tilde{y})=h(i\tilde{y},\tilde{x})$ (resp.\ $\tilde{z}=f(\tilde{x},\tilde{y})=-ih(\tilde{x},i\tilde{y})$) is a solution to (\ref{eq:M}) with the even (resp.\ odd) symmetry with respect to the $\tilde{y}$-axis. By using Theorem \ref{thm:BI and M} for the minimal graph $f$ in $\mathbb{E}^3(\tilde{x},\tilde{y},\tilde{z})$, we conclude that the Wick rotated solution of $f$, which is nothing but the original solution $h$, is timelike and has negative (resp.\ positive) Gaussian curvature. Away from flat points the diagonalizability of the shape operator of a timelike minimal surface is determined by the sign of the Gaussian curvature $K$, see \cite{A2} for details. Since the shape operator is diagonalizable over $\mathbb{R}$ (resp.\ $\mathbb{C}\setminus \mathbb{R}$) on a point where $K$ is negative (resp.\ positive), we have the desired result.
\end{proof}

 \subsection{Geometric properties of transformations between solutions to (\ref{eq:ZMC}) and (\ref{eq:BI})}
A relation between solutions to (\ref{eq:ZMC}) and (\ref{eq:BI}) was pointed out by R.\ Dey and the second author in \cite[Proposition 2.1]{DS}. Similar to Propositions \ref{prop:BI and M_even} and \ref{prop:BI and M_odd}, we can prove the following proposition.
\begin{proposition}\label{prop:BI and ZMC}
The following correspondences hold.
\begin{itemize}
\item[(i)] For a solution $t=g(x,y)$ to (\ref{eq:ZMC}) with even symmetry to the $x$-axis, $\tilde{y}=g(i\tilde{x},\tilde{t})$ is a real solution to (\ref{eq:BI}) in $\mathbb{L}^3(\tilde{t},\tilde{x},\tilde{y})$. Conversely, for a solution $\tilde{y}=h(\tilde{t},\tilde{x})$ to (\ref{eq:BI}) in $\mathbb{L}^3(\tilde{t},\tilde{x},\tilde{y})$ with even symmetry with respect to the $\tilde{x}$-axis, $t=h(y,ix)$ is a real solution to (\ref{eq:ZMC}) in $\mathbb{L}^3(t,x,y)$, where $\tilde{t}=y$, $\tilde{x}=x$ and $\tilde{y}=t$.
\item[(ii)] For a solution $t=g(x,y)$ to (\ref{eq:ZMC}) with odd symmetry to the $x$-axis, $\tilde{y}=-ig(i\tilde{t},\tilde{x})$ is a real solution to (\ref{eq:BI}) in $\mathbb{L}^3(\tilde{t},\tilde{x},\tilde{y})$. Conversely, for a solution $\tilde{y}=h(\tilde{t},\tilde{x})$ to (\ref{eq:BI}) in $\mathbb{L}^3(\tilde{t},\tilde{x},\tilde{y})$ with odd symmetry to the $\tilde{t}$-axis, $t=-ih(ix,y)$ is a real solution to (\ref{eq:ZMC}) in $\mathbb{L}^3(t,x,y)$, where $\tilde{t}=x$, $\tilde{x}=y$ and $\tilde{y}=t$.
\end{itemize}
\end{proposition}
In Theorem \ref{thm:BI and M}, we saw that minimal graphs in $\mathbb{E}^3$ correspond to timelike minimal graphs over a domain of the timelike $tx$-plane in $\mathbb{L}^3(t,x,y)$. On the other hand, by Wick rotations between solutions to (\ref{eq:ZMC}) and (\ref{eq:BI}), causal characters of solutions change in general. First we consider Wick rotations near spacelike or timelike part of solutions (In Section \ref{Sec.4}, we will discuss Wick rotations near lightlike points). Based on the correspondences given in Proposition \ref{prop:BI and ZMC}, we prove the next theorem which relates the causal characters, symmetries and Gaussian curvatures of solutions. 

\begin{theorem}\label{thm:BI and ZMC}
Let $t=g(x,y)$ be a solution to (\ref{eq:ZMC}). The following statements hold.
\begin{itemize}
\item[(i)] If $g$ is even with respect to the $x$-axis, then the graph of $g$ is spacelike (resp.\ timelike) if and only if the solution $\tilde{y}=h(\tilde{t},\tilde{x})=g(i\tilde{x},\tilde{t})$ to (\ref{eq:BI}) in $\mathbb{L}^3(\tilde{t},\tilde{x},\tilde{y})$ is timelike (resp.\ spacelike), where $\tilde{t}=y$, $\tilde{x}=x$ and $\tilde{y}=t$. Moreover the Gaussian curvature of the graph of $g$ is positive (resp.\ negative), and that of the graph of $h$ is negative (resp.\ positive) away from umbilic points.
\item[(ii)]If $g$ is odd with respect to the $x$-axis, then the graph of $g$ is spacelike (resp.\ timelike) if and only if the solution $\tilde{y}=h(\tilde{t},\tilde{x})=-ig(i\tilde{t},\tilde{x})$ to (\ref{eq:BI}) in $\mathbb{L}^3(\tilde{t},\tilde{x},\tilde{y})$ is timelike (resp.\ spacelike), where $\tilde{t}=x$, $\tilde{x}=y$ and $\tilde{y}=t$. Moreover the Gaussian curvatures of both graphs  are positive away from umbilic points.
\end{itemize}
\end{theorem}
\begin{proof}
We prove only the first assertion of (i). The rest of the proof is the same as that of Theorem \ref{thm:BI and M}. 
Since the first fundamental form of the graph of $h$ is
\[
(-1+g_y^2(i\tilde{x},\tilde{t}))d\tilde{t}^2+2ig_x(i\tilde{x},\tilde{t})g_y(i\tilde{x},\tilde{t})d\tilde{t}d\tilde{x}+(1-g_x^2(i\tilde{x},\tilde{t}))d\tilde{x}^2,
\]
the causal character of the graph of $h$ at $(\tilde{t},\tilde{x})$ is determined by its determinant, that is, the sign of $-1+g_x^2(i\tilde{x},\tilde{t})+g_y^2(i\tilde{x},\tilde{t})$.
On the other hand, the causal character of the graph of $g$ at $(x,y)$ is determined by the sign of $1-g_x^2(x,y)-g_y^2(x,y)$, hence the graph of $h$ is spacelike (resp. timelike) at $(\tilde{t},\tilde{x})=(\tilde{t},0)$ if and only if the graph of $g$ is timelike (resp.\ spacelike) at $(x,y)=(0,y)$.
\end{proof}

In addition to Corollary \ref{cor:symmetry1}, we complete the description of relations between symmetries of a timelike minimal graph and the diagonalizability of the shape operator as follows.
\begin{corollary}\label{cor:symmetry2}
The following statements hold.
\begin{itemize}
 \item[(i)]Away from flat points, a timelike minimal graph $y=h(t,x)$ with even (resp.\ odd) symmetry with respect to the $x$-axis (resp.\ $t$-axis) has real (resp.\ complex) principal curvatures near the axis.
\item[(ii)]Away from flat points, a timelike minimal graph $t=g(x,y)$ with even (resp.\ odd) symmetry with respect to the $x$-axis has real (resp.\ complex) principal curvatures  near the axis.
\end{itemize}
\end{corollary}
\begin{proof}
By Lemma \ref{lemma:symmetry}, the Wick rotation of $h$ or $g$ with respect to the considered axis is also even (or odd) and by using Theorem \ref{thm:BI and ZMC}, we have the desired result.
\end{proof}

\subsection{Geometric properties of transformations between solutions to (\ref{eq:M}) and (\ref{eq:ZMC})}
As we saw in Introduction, we can construct a solution to (\ref{eq:ZMC}) by taking Wick rotations with respect to two variables of a solution to (\ref{eq:M}) when the solution has symmetries. The following theorem is an immediate consequence of Theorems \ref{thm:BI and M} and \ref{thm:BI and ZMC}.
\begin{theorem}\label{thm:M and ZMC}
Let $z=f(x,y)$ be a solution to (\ref{eq:M}) in $\mathbb{E}^3(x,y,z)$. If $f$ is even (resp.\ odd) with respect to the $x$ and $y$ axes, then its Wick rotation $t=g(x,y)=f(ix,iy)$  (resp.\ $t=g(x,y)=-f(ix,iy)$) is a real solution to (\ref{eq:ZMC}) in $\mathbb{L}^3(t,x,y)$ which is spacelike at least near the origin $o=(0,0)$.

\end{theorem}
Conversely, if we take Wick rotations of a solution to (\ref{eq:ZMC}) which is even (or odd) with respect to two variables and spacelike at least near the origin, we obtain a real solution to (\ref{eq:M}).

\section{Transformations of zero mean curvature surfaces with a lightlike line}\label{Sec.4}
In Proposition \ref{prop:BI and ZMC} and Theorem \ref{thm:BI and ZMC}, we saw how solutions to (\ref{eq:ZMC}) and (\ref{eq:BI}) are transformed to each other near spacelike and timelike points. In this section, we study Wick rotations starting from lightlike points. As an application, we give a transformation theory of zero mean curvature surfaces which contain a lightlike line.

\subsection{Wick rotations starting from lightlike points}
In the following, we are concerned with Wick rotations starting from lightlike points. For a solution $g$ to (\ref{eq:ZMC}), we define the function $B_g=1-g_x^2-g_y^2$ and its gradient $\nabla{B_g}=((B_g)_x,(B_g)_y)$. In \cite{Klyachin}, Klyachin showed the following property of lightlike points on zero mean curvature surfaces in $\mathbb{L}^3$:
\begin{theorem}[\cite{Klyachin} and cf. also \cite{UY}]\label{thm:Klyachin}
Let $g$ be a solution to (\ref{eq:ZMC}) and $o=(0,0)$ be a lightlike point in the domain of $g$. Then either of the following holds.
\begin{itemize}
\item [(i)] For the case that $\nabla{B_g}(o)\neq0$, the image of lightlike points near $o$ is a non-degenerate null curve $\gamma$ in $\mathbb{L}^3$ across which the causal character of the graph is changed, where a non-degenerate null curve is a regular curve in $\mathbb{L}^3$ whose velocity vector field is lightlike and linearly independent to its acceleration vector field.
\item[(ii)] For the case that $\nabla{B_g}(o)=0$, the image of lightlike points near $o$ contains a lightlike line segment in $\mathbb{L}^3$ passing through $(g(o),o)$.
\end{itemize}
\end{theorem}
The first case is now well understood. In fact, for the null curve $\gamma$ in (i) of Theorem \ref{thm:Klyachin}, the spacelike part $\Phi$ and the timelike part $\Psi$ of the graph of $g$ can be written as 
\begin{equation}\label{eq:ZMCext}
  \Phi(u,v)=\frac{\gamma(u+iv)+\gamma(u-iv)}{2}\quad \text{and}\quad 
  \Psi(u,v)=\frac{\gamma(u+v)+\gamma(u-v)}{2},
\end{equation}
respectively, and the images of $\Phi$ and $\Psi$ match real analytically along $\gamma$. See \cite{FujimoriETAL3,Gu,KKSY,Klyachin} for details.
Moreover, as noted in \cite[Section 2]{KKSY}, these two parts are also related by the Wick rotations
\[
\Phi(u,iv)=\Psi(u,v)\quad \text{and}\quad \Psi(u,iv)=\Phi(u,v).
\]

On the other hand, the second case has been studied intensively in recent years \cite{A1,FujimoriETAL1,FujimoriETAL2,FujimoriETAL3,UY}. In this section, we give a transformation theory for zero mean curvature surfaces with lightlike points satisfying the condition (ii) in Theorem \ref{thm:Klyachin} via Wick rotations. First we show that lightlike lines are transformed to each other under Wick rotations on lightlike points.

\begin{lemma}Let $t=g(x,y)$ be a solution to (\ref{eq:ZMC}) with $g(o)=0$, and $o=(0,0)$ be a lightlike point satisfying $\nabla{B_g}(o)=0$. The following statements hold.
\begin{itemize}\label{Lemma:lightlikelines}
\item[(i)] If $g$ is even with respect to the $x$-axis, then there exists a lightlike line segment $L$, which lies in either $\{(y,0,y)\mid y\in \mathbb{R}\}$ or $\{(-y,0,y)\mid y\in \mathbb{R}\}$, and the Wick rotation $\tilde{y}=h(\tilde{t},\tilde{x})=g(i\tilde{x},\tilde{t})$ in (i) of Proposition \ref{prop:BI and ZMC} also has a lightlike line segment $\tilde{L}$, which lies in either $\{(\tilde{y},0,\tilde{y})\mid \tilde{y}\in \mathbb{R}\}$ or $\{(-\tilde{y},0,\tilde{y})\mid \tilde{y}\in \mathbb{R}\}$.
\item[(ii)] If $g$ is odd with respect to the $x$-axis, then there exists a lightlike line segment $L$, which lies in either $\{(x,x,0)\mid x\in \mathbb{R}\}$ or $\{(-x,x,0)\mid x\in \mathbb{R}\}$, and the Wick rotation $\tilde{y}=h(\tilde{t},\tilde{x})=-ig(i\tilde{t},\tilde{x})$ in (ii) of Proposition \ref{prop:BI and ZMC} also has a lightlike line segment $\tilde{L}$, which lies in either $\{(\tilde{y},0,\tilde{y})\mid \tilde{y}\in \mathbb{R}\}$ or $\{(-\tilde{y},0,\tilde{y})\mid \tilde{y}\in \mathbb{R}\}$.
\end{itemize}
\end{lemma}

\begin{proof}
By the even symmetry of $g$ to the $x$-axis, we have $g_{x}(0,y)=g_{xy}(0,y)=0$. Hence the assumptions $B_g(o)=0$ and $\nabla{B_g}(o)=0$ are equivalent to $g_y(o)=\pm 1$ and $g_{yy}(o)=0$. Therefore, by Theorem \ref{thm:Klyachin}, the graph of $g$ has a lightlike line $L$ whose direction is $(\pm 1,0,1)$. Since $g(o)=0$, $L$ is in either $\{(y,0,y)\mid y\in \mathbb{R}\}$ or $\{(-y,0,y)\mid y\in \mathbb{R}\}$. By the Wick rotation, $L$ is moved to the following lightlike line $\tilde{L}$ on the graph of $h$ in $\mathbb{L}^3(\tilde{t},\tilde{x},\tilde{y})$:
\[
\tilde{y}=h(\tilde{t},0)=g(0,\tilde{t})=\pm \tilde{t},
\]
where the directions of the $L$ and $\tilde{L}$ are depending on the sign of $g_y(o)=\pm1$. The proof of (ii) is same as the previous case.
\end{proof}
\begin{remark}
For the case (i) in Lemma \ref{Lemma:lightlikelines}, by the even symmetry of $g$ and the equation (\ref{eq:ZMC}), $\nabla{B_g}(o)=0$ follows automatically.
\end{remark}
For any zero mean curvature surface containing a lightlike line segment $L$, we may assume that $L$ is in $\{(y,0,y)\mid y\in \mathbb{R}\}$. Locally such surface is represented as $t=g(x,y)$ near $L$. Near $L$, we can expand the function $g$ as
\begin{equation}
g(x,y)=y+\frac{\alpha_g(y)}{2}x^2+\beta_g(x,y)x^3,
\end{equation}
where $\alpha_g=\alpha_g(y)$ and $\beta_g=\beta_g(x,y)$ are real analytic functions. In \cite{UY}, the function $\alpha_g$ is called the {\it (second) approximation function} of the graph of $g$, and $\alpha_g$ can be written as $\alpha_g(y)=g_{xx}(0,y)$. The approximation function $\alpha_g$ satisfies the differential equation
\begin{equation}\label{eq:alpha}
\frac{d\alpha_g}{dy}(y)+\alpha_g ^2(y)+\mu =0,
\end{equation}
where $\mu $ is a real constant called the {\it characteristic} along $L$, see \cite{FujimoriETAL1}. By a homothetic change, we can normalize $\mu$ to be $1$, $0$, or $-1$, and $\alpha_g$ is one of the following explicit solutions to (\ref{eq:alpha}) depending on $\mu=1$, $0$ or $-1$:
 \begin{align}
 & \alpha^+:=-\tan(y+c)\quad (|c|<\frac{\pi}{2})\quad \text{for $\mu=1$},& \nonumber\\
 & \alpha^0_{\rm I}:=0,\quad \alpha^0_{\rm II}:=(y+c)^{-1}\quad (c\neq 0)\quad  \text{for $\mu=0$},& \nonumber\\
 & \alpha^{-}_{\rm I}:=\tanh(y+c)\quad (c\in \mathbb{R}),\quad  \alpha^{-}_{\rm II}:=\coth(y+c)\quad (c\neq0),\quad \alpha^{-}_{\rm III}:=\pm 1\quad \text{for $\mu=-1$}.&\nonumber
 \end{align}
Therefore all zero mean curvature surfaces containing a lightlike line are categorized into the above six classes. In \cite{A1,FujimoriETAL1,FujimoriETAL2,FujimoriETAL3}, many important examples of zero mean curvature surfaces with lightlike lines were constructed, and the types of $\alpha_g$ of these examples were determined. On the causal character near the lightlike line $L$, the following property is known.
\begin{proposition}[\cite{FujimoriETAL1}]\label{Prop:FujimoriETAL1}
If $\mu>0$ (resp.\ $\mu<0$), the surface is spacelike (resp.\ timelike) on both-sides of $L$. On the other hand, if $\mu=0$, the causal characters of surfaces near $L$ need not be unique.
\end{proposition}

Based on Proposition \ref{prop:BI and ZMC} and Lemma \ref{Lemma:lightlikelines}, types of the approximation functions of zero mean curvature surfaces containing lightlike lines are transformed via Wick rotations as follows.

\begin{theorem}\label{thm:transformations of L-line}
Let $t=g(x,y)$ be a solution to (\ref{eq:ZMC}) as in Lemma \ref{Lemma:lightlikelines} with the approximation function $\alpha_g$ along a lightlike line segment $L$ on the graph of $g$.
\begin{itemize}
\item[(i)] If $g$ is even with respect to the $x$-axis, then the solution $\tilde{y}=h(\tilde{t},\tilde{x})=g(i\tilde{x},\tilde{t})$ 
has the approximation function $\alpha_h=\alpha_g$ along $\tilde{L}$ as in (i) of Lemma \ref{Lemma:lightlikelines}.
\item[(ii)] If $g$ is odd with respect to the $x$-axis, then the solution $\tilde{y}=h(\tilde{t},\tilde{x})=-ig(i\tilde{t},\tilde{x})$ has the approximation function $\alpha_h=i(\alpha_g\circ i)$ along $\tilde{L}$ as in (ii) of Lemma \ref{Lemma:lightlikelines}. Moreover each of $\alpha_g$ or $\alpha_h$ is an odd function, and it is one of $\alpha^+$, $\alpha^0_{I}$ or $\alpha^-_{I}$. A graph of type $\alpha^+$ (resp.\ $\alpha^-_{I}$) is transformed to a graph of type $\alpha^-_{I}$ (resp.\ $\alpha^+$), and a graph of type $\alpha^0_{I}$ is transformed to a graph of type $\alpha^0_{I}$.
\end{itemize}
\end{theorem}

\begin{proof}
Near the lightlike line $\tilde{L}$, we can write the graph of $h$ as the graph of the $\tilde{x}\tilde{y}$-plane of a function $f=f(\tilde{x},\tilde{y})$ in $\mathbb{L}^3(\tilde{t},\tilde{x},\tilde{y})$. First we prove (i). By Lemma \ref{Lemma:lightlikelines}, we may assume that $L\subset \{(y,0,y)\mid y\in \mathbb{R}\}$ and  $\tilde{L} \subset \{(\tilde{y},0,\tilde{y})\mid \tilde{y}\in \mathbb{R}\}$. The approximation functions $\alpha_g$ and $\alpha_h$ are written as
\[
\alpha_g(y)=g_{xx}(0,y),\quad \alpha_h(\tilde{y})=f_{\tilde{x}\tilde{x}}(0,\tilde{y}).
\]
Taking the derivative of the equation $\tilde{y}=h(f(\tilde{x},\tilde{y}),\tilde{x})$ with respect to $\tilde{x}$, we have
\begin{equation}\label{eq:f_x}
0=h_{\tilde{t}}f_{\tilde{x}}+h_{\tilde{x}}.
\end{equation}
Since the graph of $g$ contains the lightlike line segment $L$ in $\{(y,0,y)\mid y\in \mathbb{R}\}$, $h_{\tilde{x}}(\tilde{t},0)=0$ and $h_{\tilde{t}}(\tilde{t},0)=1$. Hence we have $f_{\tilde{x}}(0,\tilde{y})=0$ on $\tilde{L}$ by (\ref{eq:f_x}). Moreover, by taking the derivative of (\ref{eq:f_x}) with $\tilde{x}$ again,
\begin{equation}\label{eq:f_xx}
0=h_{\tilde{t}\tilde{t}}f_{\tilde{x}}^2+2h_{\tilde{t}\tilde{x}}f_{\tilde{x}}+h_{\tilde{t}}f_{\tilde{x}\tilde{x}}+h_{\tilde{x}\tilde{x}}.
\end{equation}
Since $f_{\tilde{x}}(0,\tilde{y})=0$, $h_{\tilde{t}}(\tilde{y},0)=1$ and $h_{\tilde{x}\tilde{x}}(\tilde{y},0)=-\alpha_g(\tilde{y})$ on $\tilde{L}$, we obtain $f_{\tilde{x}\tilde{x}}(0,\tilde{y})=\alpha_g(\tilde{y})$ by (\ref{eq:f_xx}). Therefore we have $\alpha_h=\alpha_g$. 

Next we prove (ii). By (ii) of Lemma \ref{Lemma:lightlikelines}, we may assume that $L\subset \{(x,x,0)\mid x\in \mathbb{R}\}$ and  $\tilde{L} \subset \{(\tilde{y},0,\tilde{y})\mid \tilde{y}\in \mathbb{R}\}$. The approximation functions $\alpha_g$ and $\alpha_h$ are written as
\[
\alpha_g(x)=g_{yy}(x,0),\quad \alpha_h(\tilde{y})=f_{\tilde{x}\tilde{x}}(0,\tilde{y}).
\]
Since the graph of $h$ contains the lightlike line $\tilde{L}$ in $\{(\tilde{y},0,\tilde{y})\mid \tilde{y}\in \mathbb{R}\}$, $h_{\tilde{t}}(\tilde{t},0)=1$ and $h_{\tilde{x}}(\tilde{t},0)=0$. Hence we have $f_{\tilde{x}}(0,\tilde{y})=0$ on $\tilde{L}$ by (\ref{eq:f_x}). The equation (\ref{eq:f_xx}) becomes
\[
0=f_{\tilde{x}\tilde{x}}(0,\tilde{y})+h_{\tilde{x}\tilde{x}}(\tilde{y},0) \text{ on $\tilde{L}$.}
\]
Since $h_{\tilde{x}\tilde{x}}(\tilde{t},0)=-i\alpha_g(i\tilde{t})$ and $\alpha_h(\tilde{y})=f_{\tilde{x}\tilde{x}}(0,\tilde{y})$, we obtain $\alpha_h(\tilde{y})=i\alpha_g(i\tilde{y})$. Moreover, by the symmetry of $g$ with respect to the $x$-axis, $\alpha_g$ and $\alpha_h$ are odd functions. Since $\alpha^+=-\tan{y}$, $\alpha^0_{I}=0$ or $\alpha^-_{I}=\tanh{y}$, here integral constants are determined by the odd symmetry of $g$ automatically, are the only approximation functions with the odd symmetry in explicit solutions to (\ref{eq:alpha}), we obtain the desired result.
\end{proof}
 By Proposition \ref{Prop:FujimoriETAL1}, we have the following corollary.
 
\begin{corollary}\label{Cor:causality_near_lines}
Let $g$ be same as in Theorem \ref{thm:transformations of L-line}. 
If $g$ is even (resp.\ odd) with respect to the $x$-axis, then the Wick rotation with respect to the $x$-axis preserves (resp.\ changes) the causal characters of graphs near lightlike lines, except for $\alpha^0_{I}$ and $\alpha^0_{II}$ types.
\end{corollary}

\section{Examples}\label{Sec.5}
In this section, we give examples of minimal surfaces in $\mathbb{E}^3$ and zero mean curvature surfaces in $\mathbb{L}^3$, and explain how these examples are related to each other via Wick rotations. 
\begin{example}\label{Ex:5.1} By Wick rotations, we can make many correspondences among catenoids in $\mathbb{E}^3$ and $\mathbb{L}^3$ as follows. First let us consider the upper half part of the catenoid in $\mathbb{E}^3$ given by $\cosh^2{z}=x^2+y^2$, which is written as $z=f(x,y)=\arccosh{(\sqrt{x^2+y^2})}$. Since the function $f$ is even with respect to $y$, we can take the Wick rotation
\[
\tilde{y}=h(\tilde{t},\tilde{x})=f(\tilde{x},i\tilde{t})=\arccosh{(\sqrt{\tilde{x}^2-\tilde{t}^2})}\quad \text{in $\mathbb{L}^3(\tilde{t},\tilde{x},\tilde{y})$}.
\]
By (i) of Theorem \ref{thm:BI and M}, it is a timelike minimal surface with negative Gaussian curvature. The graph of $h$ has the implicit form $\cosh^2{\tilde{y}}=\tilde{x}^2-\tilde{t}^2$, which is called the {\it timelike hyperbolic catenoid of type I}. Next if we rotate the catenoid in $\mathbb{E}^3$ as $\cosh^2{y}=x^2+z^2$, we obtain the graph $z=\sqrt{\cosh^2{y}-x^2}$, which is also even with respect to $y$. By (i) of Theorem \ref{thm:BI and M}, its Wick rotation
\[
\tilde{y}=\sqrt{\cos^2{\tilde{t}}-\tilde{x}^2}\quad \text{in $\mathbb{L}^3(\tilde{t},\tilde{x},\tilde{y})$}
\]
is also a timelike minimal surface with negative Gaussian curvature. This surface has the implicit form $\tilde{y}^2=\cos^2{\tilde{t}}-\tilde{x}^2$ called the {\it timelike elliptic catenoid}. On the other hand, if we take the Wick rotation of this surface with respect to $\tilde{x}$, we obtain the surface $t=\sqrt{\cos^2{y}+x^2}$ called the {\it spacelike hyperbolic catenoid}. Finally, if we start the {\it spacelike elliptic catenoid} $y=\sqrt{\sinh^2{t}-x^2}$, we obtain its Wick rotation with respect to $x$
\[
\tilde{t}=\sqrt{\sinh^2{\tilde{y}+\tilde{x}^2}},
\]
which is known as the {\it timelike hyperbolic catenoid of type II} (see Figure \ref{Fig5_1}). About the names of catenoids in $\mathbb{L}^3$, see \cite{KKSY,Lopez} for details.
\end{example}
\begin{figure}
\begin{center}
\begin{tabular}{c}
\begin{minipage}{0.3\hsize}
\begin{center}
\vspace{-2.5cm}
\hspace{-0.8cm}
\includegraphics[clip,scale=0.42,bb=0 0 350 380]{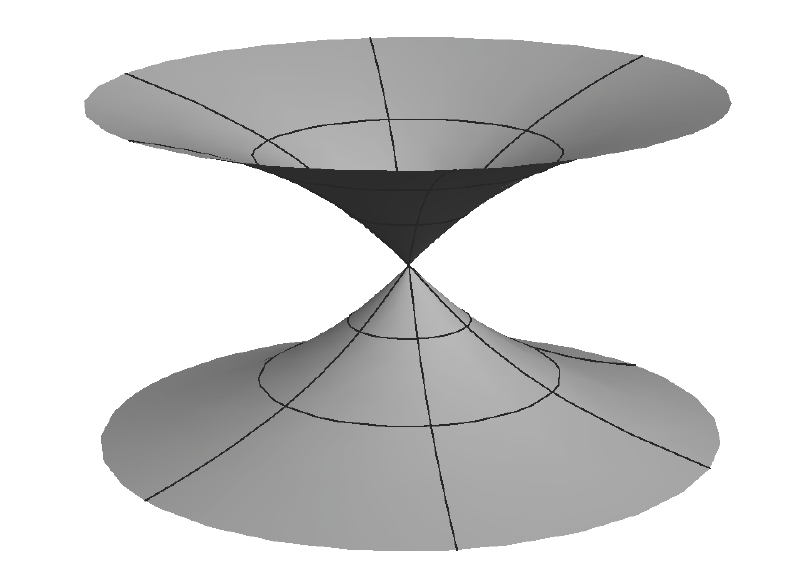}
\end{center}
\end{minipage}
\hspace{-1.9cm}
\begin{minipage}{0.3\hsize}
\begin{center}
\vspace{-1.0cm}
\includegraphics[clip,scale=0.42,bb=0 0 350 380]{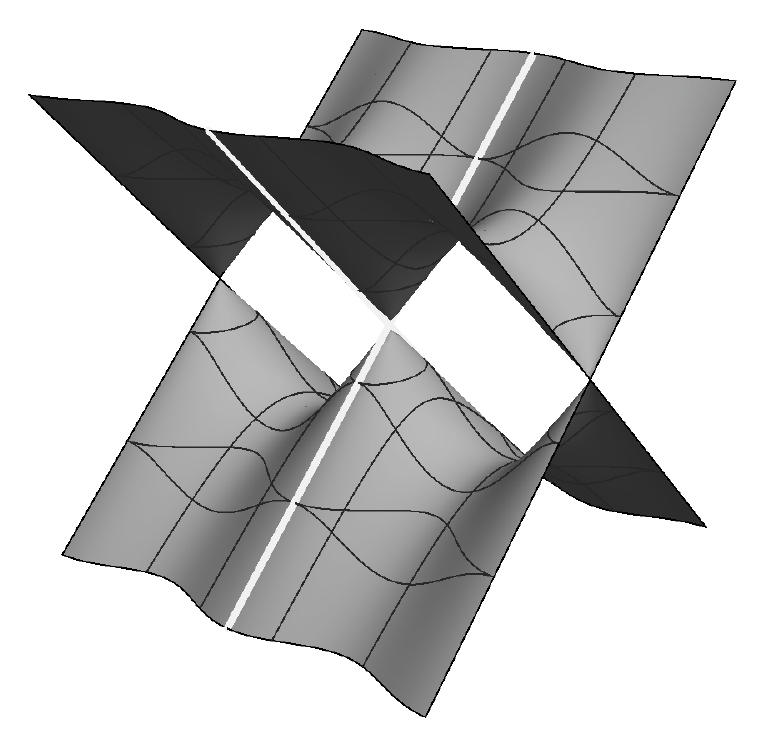}
\vspace{0.3cm}
\end{center}
\end{minipage}
\hspace{-1.0cm}
\begin{minipage}{0.3\hsize}
\begin{center}
\vspace{-0.5cm}
\includegraphics[clip,scale=0.42,bb=0 0 350 380]{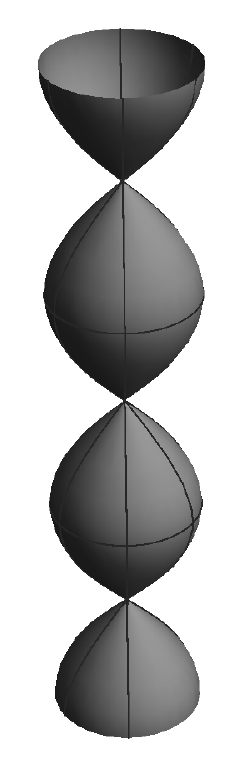}
\vspace{0.5cm}
\end{center}
\end{minipage}
\hspace{-3.0cm}
\begin{minipage}{0.3\hsize}
\begin{center}
\vspace{-1.0cm}
\includegraphics[clip,scale=0.42,bb=0 0 350 380]{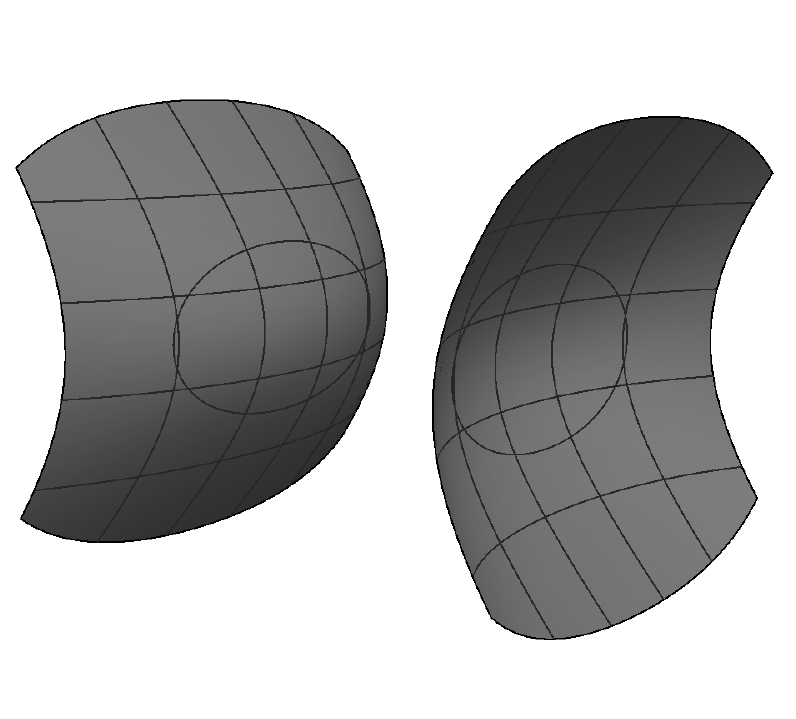}
\vspace{0.3cm}
\end{center}
\end{minipage}
\hspace{-1.3cm}
\begin{minipage}{0.3\hsize}
\begin{center}
\vspace{-0.8cm}
\includegraphics[clip,scale=0.42,bb=0 0 350 380]{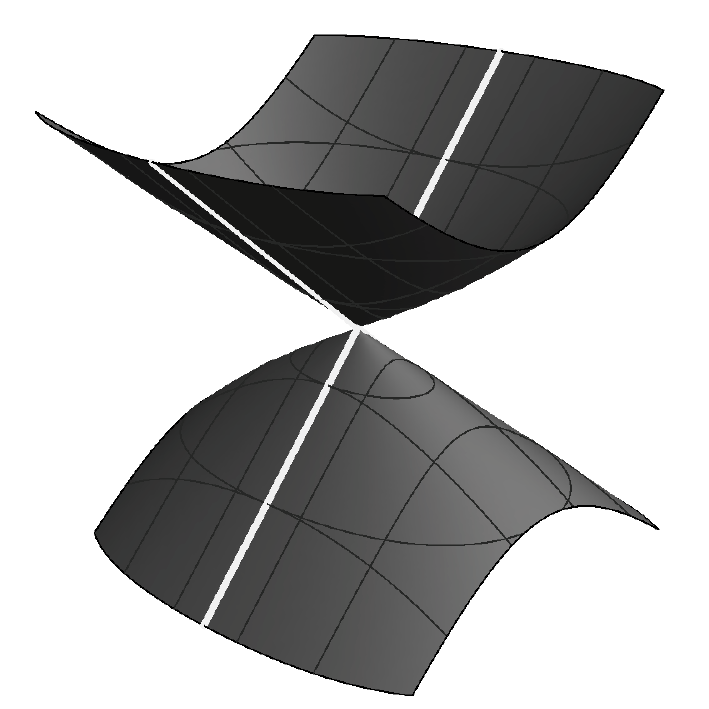}
\vspace{0.5cm}
\end{center}
\end{minipage}

\end{tabular}
\end{center}
\vspace{-0.5cm}
\caption{From left to right: the spacelike elliptic catenoid, the spacelike hyperbolic catenoid, the timelike elliptic catenoid, the timelike hyperbolic catenoid of type I and the timelike hyperbolic catenoid of type II.}\label{Fig5_1}

\end{figure}
As the above examples show, Wick rotations and isometries on the ambient space do not commute in general. By using this, we can construct many solutions to (\ref{eq:M}), (\ref{eq:ZMC}) and (\ref{eq:BI}) starting from a given solution.
\begin{example}\label{Ex:5.2}
The entire zero mean curvature graph $t=g(x,y)=x\tanh{y}$, which was discovered by Kobayashi \cite{Kobayashi} is a solution to (\ref{eq:ZMC}). Since it is odd with respect to $x$ and $y$, we can take the Wick rotation
\[
z=f(x,y)=-g(ix,iy)=x\tan{y}
\]
by Theorem \ref{thm:M and ZMC}, which is nothing but the helicoid in $\mathbb{E}^3$.
\end{example}

\begin{example}\label{Ex:5.3}
In Introduction, we saw a correspondence between the doubly periodic Scherk surface in $\mathbb{E}^3$ and Scherk type zero mean curvature surface in $\mathbb{L}^3$. Here we start from the singly periodic Scherk minimal surface $\sin{z}=-\sinh{x}\sinh{y}$, which is also called {\it Scherk saddle tower} (see Figure \ref{Fig:5.3}). Locally it can be written as $z=f(x,y)=-\arcsin{(\sinh{x}\sinh{y})}$, which is odd with respect to $x$ and $y$ axes. By Theorem \ref{thm:M and ZMC}, its Wick rotation
\[
t=g(x,y)=-f(ix,iy)=-\arcsin{(\sin{x}\sin{y})}
\]
is spacelike surface. This surface has the triply periodic implicit form $\sin{t}=\sin{x}\sin{y}$, which is called the {\it spacelike Scherk surface} in \cite[Example 3]{FujimoriETAL1}.

\begin{remark}
In contrast with Scherk surfaces in Introduction and Example \ref{Ex:5.3}, as explained in \cite[Example 1]{Lee}, the doubly periodic Scherk minimal surface corresponds to the spacelike Scherk surface, and the singly periodic Scherk minimal surface corresponds to the Scherk type zero mean curvature graph (\ref{eq:ZMC_Scherk}) by the Calabi's correspondence in Introduction.
\end{remark}
\begin{figure}[!h]
\vspace{-0.3cm}
\begin{center}
\begin{tabular}{c}
\hspace{+0.1cm}
\begin{minipage}{0.4\hsize}
\begin{center}
\vspace{-0.8cm}
\includegraphics[clip,scale=0.35,bb=0 0 350 380]{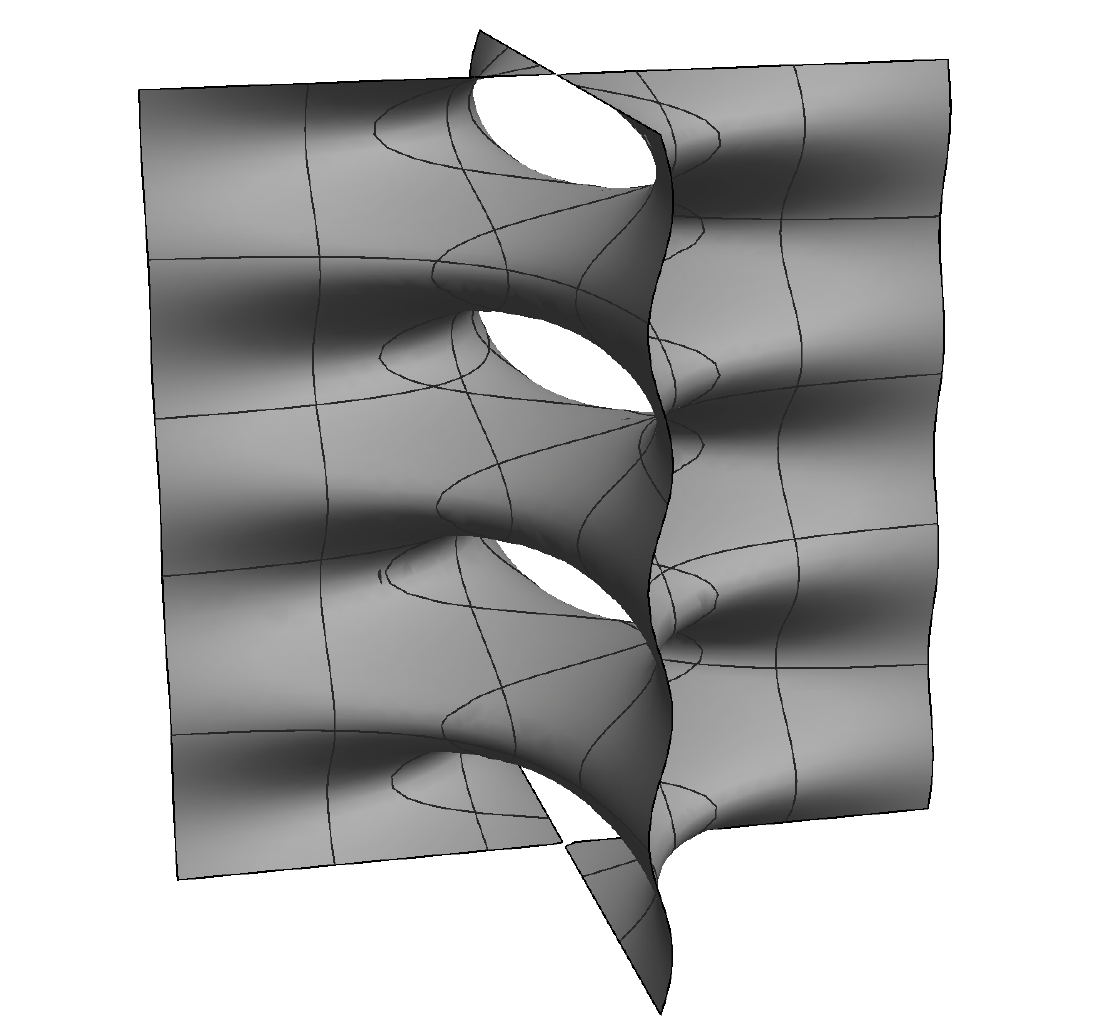}
\vspace{0.5cm}
\end{center}
\end{minipage}
\hspace{0.5cm}
\begin{minipage}{0.4\hsize}
\begin{center}
\vspace{-1.0cm}
\includegraphics[clip,scale=0.35,bb=0 0 350 380]{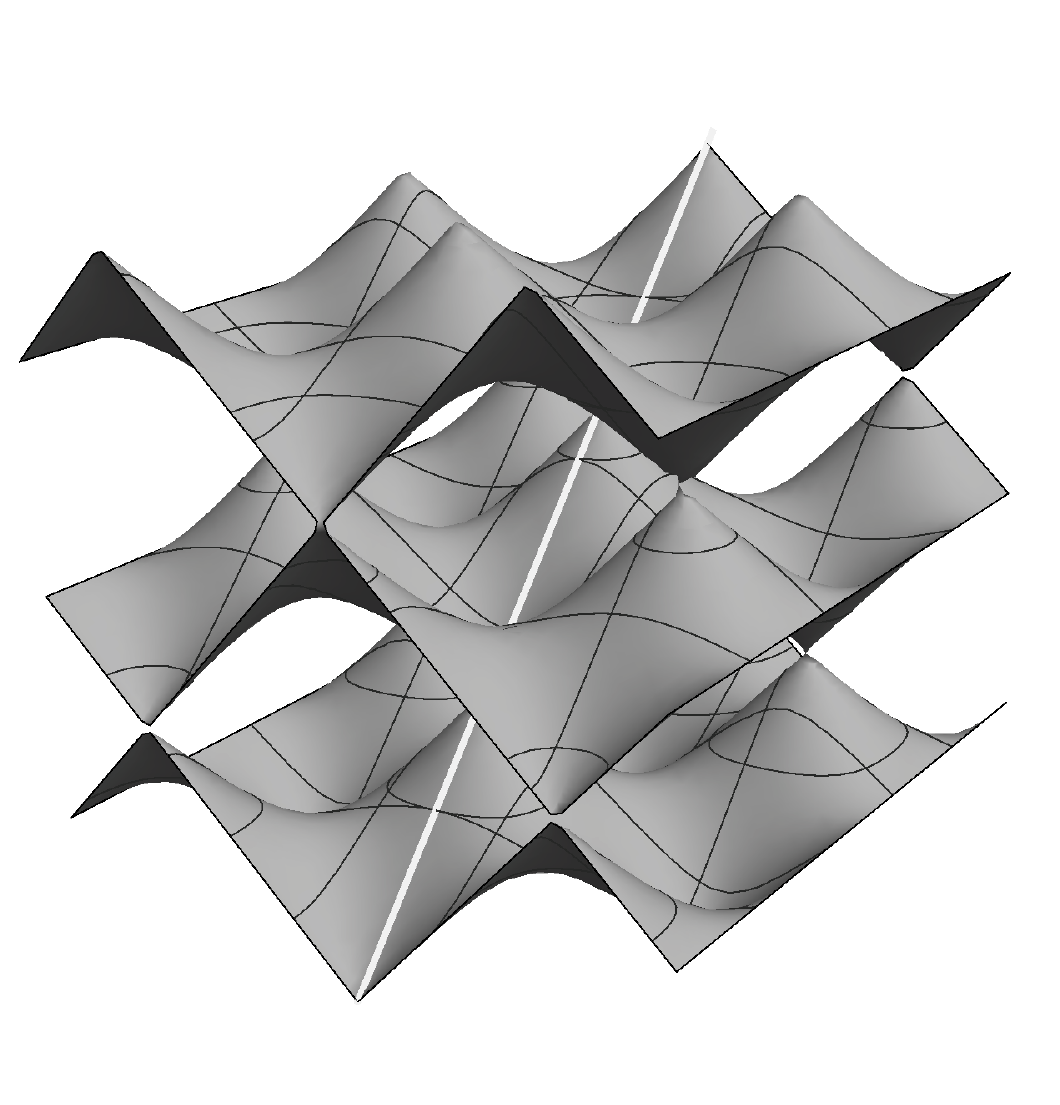}
\vspace{0.3cm}
\end{center}
\end{minipage}

\end{tabular}
\end{center}
\vspace{-0.3cm}
\caption{The singly periodic Scherk minimal surface (the left figure) and spacelike Scherk surface which have lightlike lines (the right figure).}\label{Fig:5.3}

\end{figure}

\end{example}

At the end of this section, we give examples of zero mean curvature surfaces containing lightlike lines. In particular, we reveal new relationships among some zero mean curvature surfaces with symmetries along lightlike lines, most of them were constructed in \cite{FujimoriETAL1}, by Wick rotations.  

\begin{example}\label{Ex:5.4}
The spacelike hyperbolic catenoid given by $t^2=\sin^2{x}+y^2$ and timelike hyperbolic catenoid given by $t^2=\sinh^2{x}+y^2$ in Example \ref{Ex:5.1} are have lightlike lines (see Figure \ref{Fig5_1}). Both of these surfaces are surfaces of $\alpha^0_{II}$ type along the lightlike lines $\{(y,0,y)\mid x\in \mathbb{R}\}$ and $\{(-y,0,y)\mid x\in \mathbb{R}\}$ (see \cite[Example 2]{FujimoriETAL1}). If we translate the spacelike hyperbolic catenoid, and take a graph $t=g(x,y)=\sqrt{\sin^2{x}+(y+1)^2}-1$, the function $g$ is even with respect to $x$. Therefore we can take the Wick rotation
\[ 
\tilde{y}=h(\tilde{t},\tilde{x})=g(i\tilde{x},\tilde{t})=\sqrt{-\sinh^2{\tilde{x}}+(\tilde{t}+1)^2}-1,
\]
which is nothing but the timelike hyperbolic catenoid in $\mathbb{L}^3(\tilde{t},\tilde{x},\tilde{y})$. Since these surfaces are transformed each other, they share the same approximation function by (i) of Theorem \ref{thm:transformations of L-line}.
\end{example}

\begin{example}\label{Ex:5.4}The spacelike Scherk surface $\sin{t}=\cos{x}\sin{y}$ in Example \ref{Ex:5.3} has lightlike lines. Locally this surface can be written as $t=g(x,y)=\arcsin{(\cos{x}\sin{y})}$, and has a lightlike line segment $L\subset \{(y,0,y)\mid x\in \mathbb{R}\}$. Along $L$, this surface is of $\alpha^+$ type (see \cite[Example 3]{FujimoriETAL1}). By the even symmetry of $g$ with respect to $x$-axis, its Wick rotation
\[ 
\tilde{y}=h(\tilde{t},\tilde{x})=g(i\tilde{x},\tilde{t})=\arcsin{(\cosh{\tilde{x}}\sin{\tilde{t}})}
\]
is also a maximal graph in $\mathbb{L}^3(\tilde{t},\tilde{x},\tilde{y})$ of $\alpha^+$ type along a lightlike line segment $\tilde{L} \subset \{(\tilde{y},0,\tilde{y})\mid \tilde{y}\in \mathbb{R}\}$ by (i) of Theorem \ref{thm:transformations of L-line} and Corollary \ref{Cor:causality_near_lines}. This surface has the doubly periodic implicit form $\sin{\tilde{y}}=\cosh{\tilde{x}}\sin{\tilde{t}}$ (see Figure \ref{Fig.5.4}, left). 

On the other hand, if we translate the triply periodic spacelike Scherk surface as $\sin{t}=\sin{x}\cos{y}$, which is written as $t=g(x,y)=\arcsin{(\sin{x}\cos{y})}$ locally. By the odd symmetry of $g$ with respect to $x$-axis, its Wick rotation
\[ 
\tilde{y}=h(\tilde{t},\tilde{x})=-ig(i\tilde{t},\tilde{x})=\arcsinh{(\sinh{\tilde{t}}\cos{\tilde{x}})}
\]
is an $\alpha^-_{I}$ type entire timelike minimal graph in $\mathbb{L}^3(\tilde{t},\tilde{x},\tilde{y})$ with a lightlike line segment $\tilde{L} \subset \{(\tilde{y},0,\tilde{y})\mid \tilde{y}\in \mathbb{R}\}$ by (ii) of Theorem \ref{thm:transformations of L-line} and Corollary \ref{Cor:causality_near_lines}. This surface has the singly periodic implicit form $\sinh{\tilde{y}}=\sinh{\tilde{t}}\cos{\tilde{x}}$ (see Figure \ref{Fig.5.4}, center). Moreover, since the above $h$ is even with respect to $\tilde{x}$, we can take the Wick rotation
\[
t=h(y,ix)=\arcsinh{(\sinh{y}\cosh{x})} \text{ in $\mathbb{L}^3(t,x,y)$}.
\]
This surface is called the {\it timelike Scherk surface of 2nd kind} in \cite[Example 5]{FujimoriETAL1}. By (i) of Theorem \ref{thm:transformations of L-line}, it is also an entire timelike minimal graph of $\alpha^-_{I}$ type along the lightlike line $\{(y,0,y)\mid x\in \mathbb{R}\}$.

\noindent
\textbf{Acknowledgement.} 
This study was initiated during the authors stay at University of Granada in April 2017. They would like to thank Professor Rafael L\'{o}pez for his invitation and hospitality. The first author is supported by Grant-in-Aid for JSPS Fellows Number 15J06677.

\begin{figure}[!h]
\begin{center}
\begin{tabular}{c}

\hspace{-0.8cm}
\begin{minipage}{0.4\hsize}
\begin{center}
\vspace{-1.0cm}
\includegraphics[clip,scale=0.35,bb=0 0 350 380]{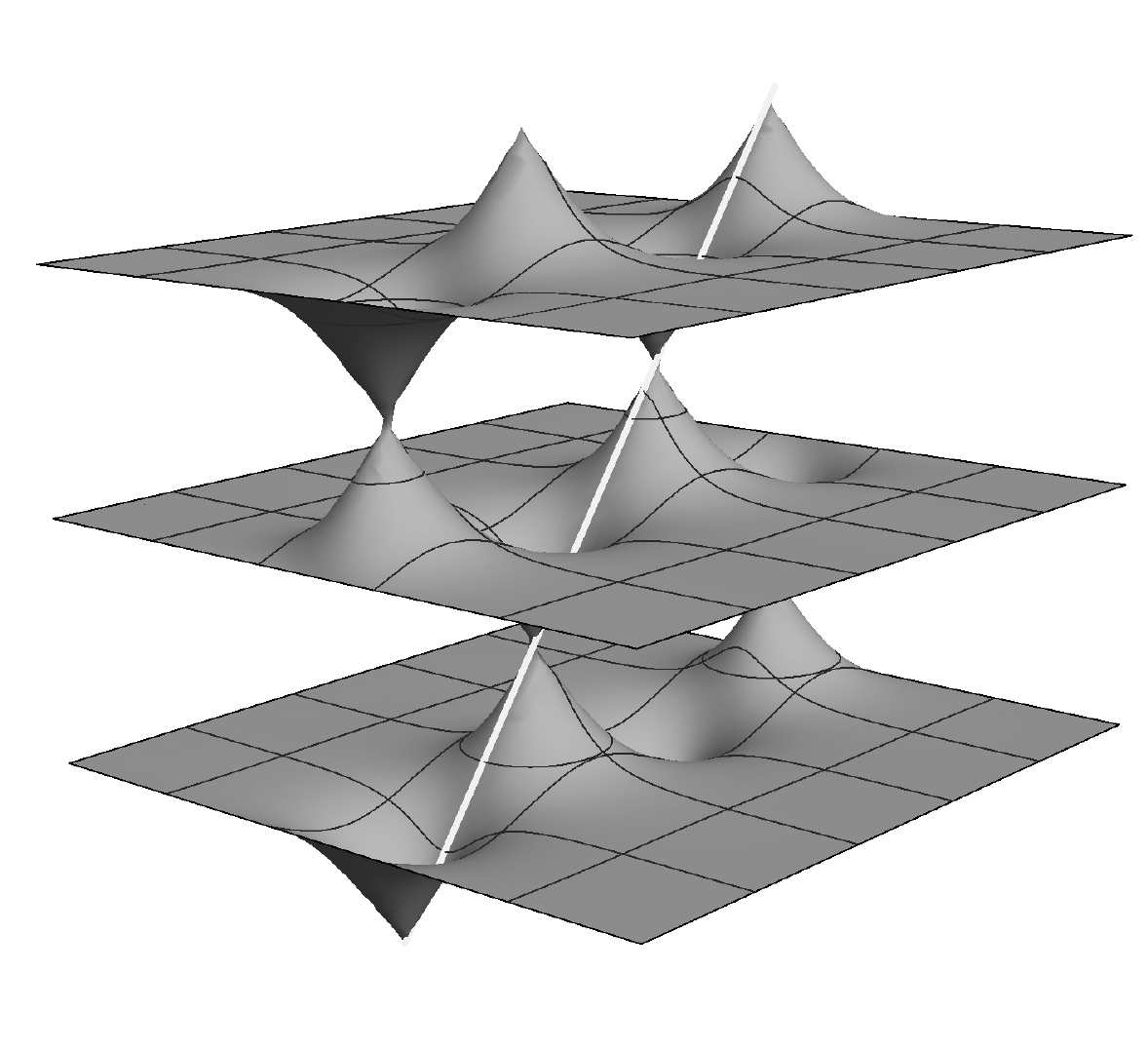}
\vspace{0.3cm}
\end{center}
\end{minipage}
\hspace{-1.0cm}
\begin{minipage}{0.4\hsize}
\begin{center}
\vspace{-0.8cm}
\includegraphics[clip,scale=0.33,bb=0 0 350 380]{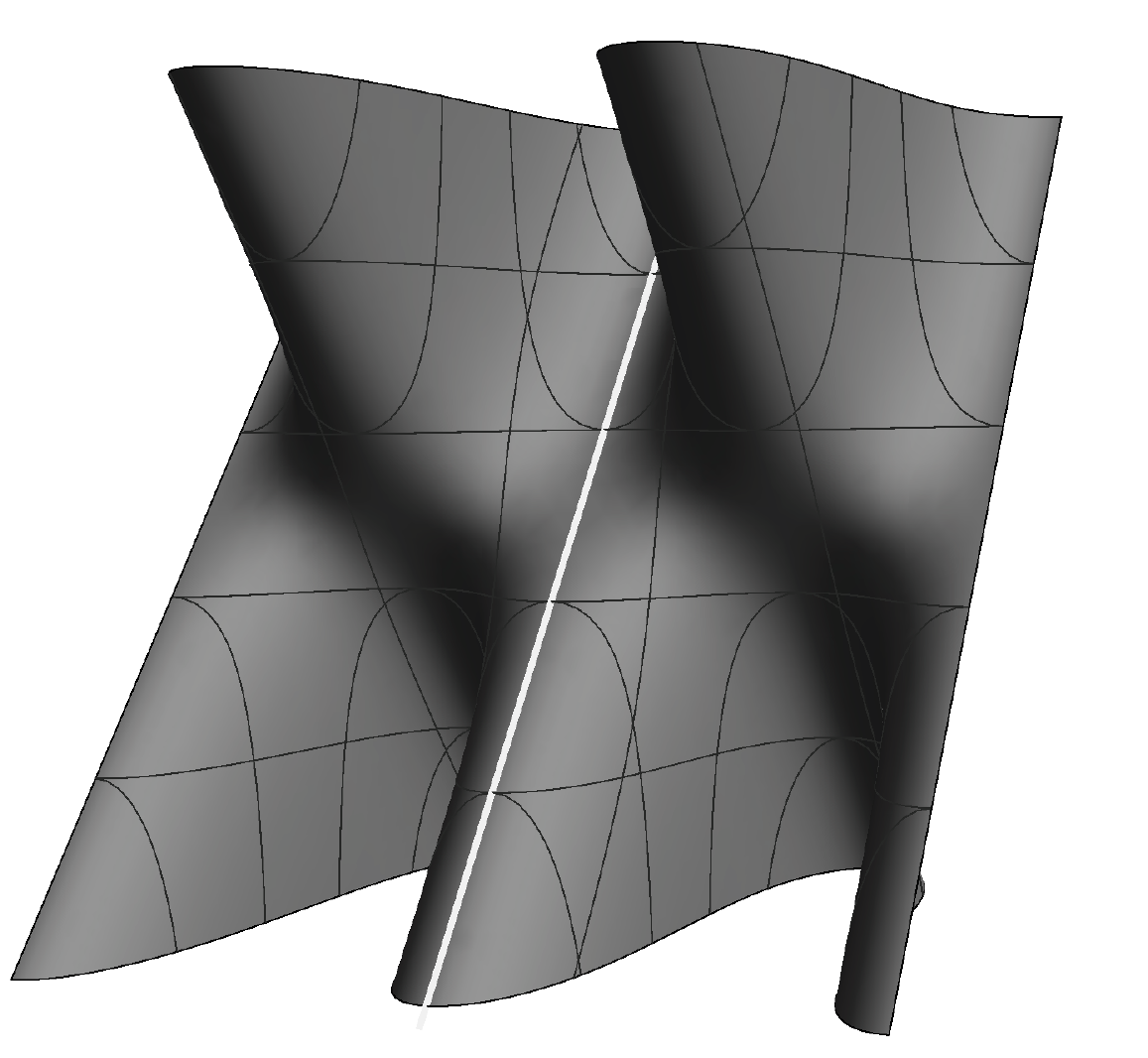}
\vspace{0.5cm}
\end{center}
\end{minipage}

\hspace{-0.8cm}
\begin{minipage}{0.4\hsize}
\begin{center}
\vspace{-1.6cm}
\hspace{-1.0cm}
\includegraphics[clip,scale=0.37,bb=0 0 350 380]{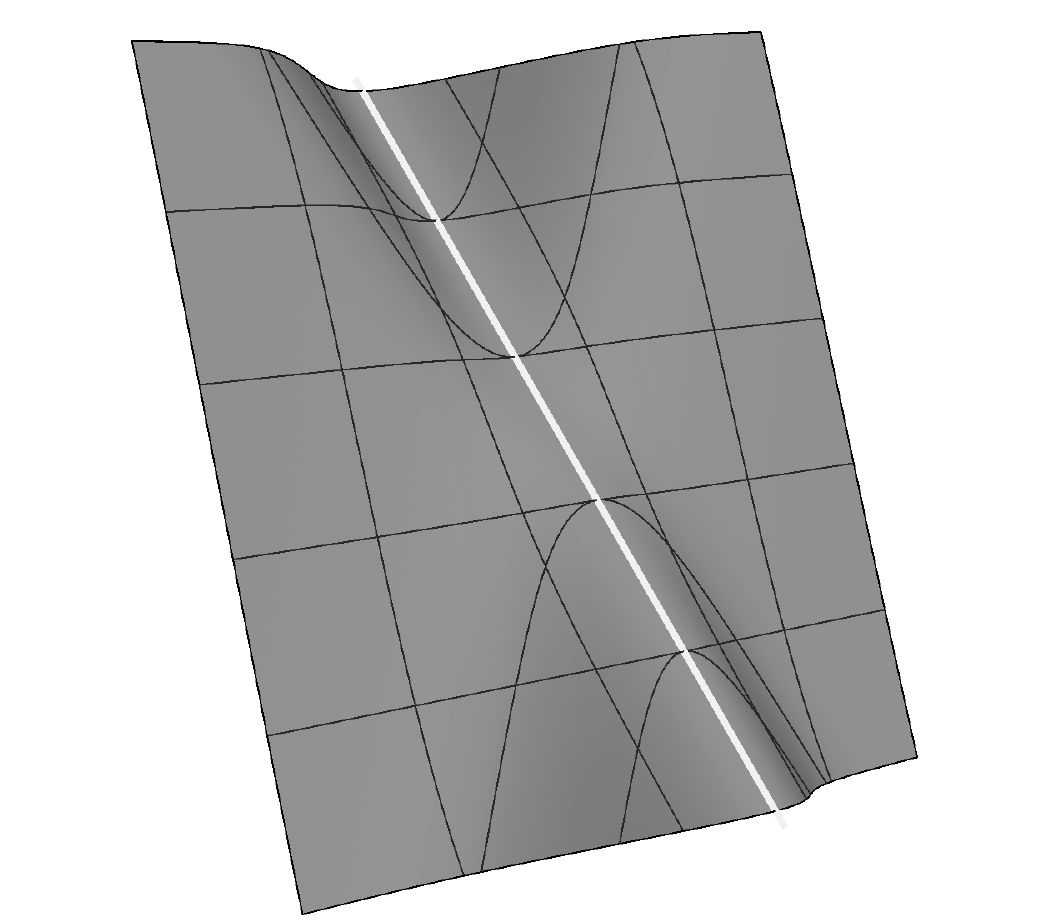}
\end{center}
\end{minipage}
\end{tabular}
\end{center}
\vspace{-0.3cm}
\caption{From left to right: the spacelike doubly periodic Scherk surface, the timelike singly periodic Scherk surface and the timelike Scherk surface of 2nd kind with a lightlike line.}\label{Fig.5.4}

\end{figure}
\end{example}


 \end{document}